\RequirePackage{fix-cm}
\documentclass[review]{svjour3}                     
\smartqed  
\usepackage{graphicx}
\usepackage[draft]{changes}  

\usepackage{amsmath}
\usepackage{amsfonts}
\usepackage{amssymb}

\usepackage{booktabs} 
\usepackage{multirow}
\usepackage{extarrows}
\usepackage{matlab-prettifier}

\usepackage{bbm} 
\usepackage{bm}
\usepackage{url}
\usepackage{enumerate}

\usepackage{ulem}

\begin{document}

\title{On the Conjecture of Stability Preservation in Arbitrary-Order Adams-Bashforth-Type Integrators
}

\titlerunning{limiting stability conjecture}        

\author{Daopeng Yin         
\and
Liquan Mei 
}

\authorrunning{D. Yin \and L. Mei} 

\institute{D. Yin \at
School of Mathematics and Statistics. Xi'an Jiaotong University. No.28, West Xianning Road, Xi'an, Shaanxi, 710049, P.R. China \\
              \email{mathdpyin@163.com}           
           \and
           L. Mei \at
School of Mathematics and Statistics. Xi'an Jiaotong University. No.28, West Xianning Road, Xi'an, Shaanxi, 710049, P.R. China
\emph{Corresponding author}: \email{lqmei@mail.xjtu.edu.cn}
}

\date{Received: date / Accepted: 2026.03.09}

\maketitle

\begin{abstract}
This paper presents stability and accuracy analysis of a high-order explicit time stepping scheme introduced by \cite[Section 2.2]{Buvoli2019}, which exhibits superior stability compared to classical Adams-Bashforth. A conjecture that is supported by several numerical phenomena in \cite[Figure 2.5]{Buvoli2018}, the method appears to remain stable when the accuracy approaches infinity, although it is not yet proven. We have disproven this conjecture from the perspective of harmonic analysis in this work. Notwithstanding the aforementioned, this method displays considerably enhanced stability in comparison to conventional explicit schemes.  Furthermore, we present a criterion for ascertaining the maximum permissible accuracy for a given specific parabolic stability radius. Conversely, the original method will lose one order associated with the expected accuracy, which can be explored theoretically.  Consequently, a unified analysis strategy for the \( L^2 \)-stability will be presented for extensional PDEs under the CFL condition. Finally, a selection of representative numerical examples will be shown in order to substantiate the theoretical analysis.  
\keywords{Adams-Bashforth-type integrator \and region of stability \and  \( L^2 \)-stability \and CFL condition.}
 \subclass{65L20\and 65M15\and 65E99}
\end{abstract}
\section{Introduction}
When dealing with time evolution equations, explicit time-discretization schemes offer advantages such as high computational efficiency and clearer uniqueness of numerical solutions. Particularly, these algorithms have remarkable advantage for the non-linear evolutionary equations.  However, their algorithmic stability has consistently hindered the promotion of higher-order explicit schemes. 
\par 
In literature \cite{Dahlquist1963}, the concept of Dahlquist’s stability barrier was introduced, which says  that there exists no A-stable explicit \( k \)-step method. Historically, numerous scholars have devoted considerable effort to enhancing the stability of explicit schemes. M. Ghrist, et al  consider a variations of the Adams–Bashforth (AB) via staggered grid techniques which improve the range of the stability area on the imaginary axis shown by \cite[Table 4.2, Table 4.3]{Ghrist2000}.   Although there are some methods called explicit ways to achieve A-stability \cite{Norsett1969,Chollom2012}, there are some discrepancies in the essence of the classic A-stability description.   Recently, a review article \cite{Givoli2023} illustrates  two types of barrier breakers.  The existing articles essentially only improve the stability of explicit schemes and cannot strictly achieve A-stability. Therefore, it is also necessary to examine the conditionally stable behavior of explicit schemes.
\par
A recently published article \cite{Buvoli2019} has proposed an intriguing algorithm that employs the Taylor expansion, the Cauchy integral formula on the complex time plane \cite{Henrici1993}, and its discretization using the trapezoidal rule \cite{Austin2014,Trefethen2014}. This algorithm is outlined in detail in reference \cite[Section 2]{Buvoli2019} of the cited article. 
In section \ref{sec:Methodological review}, we will elaborate on the connection between this scheme and the AB method. For this reason, this method will be referred to in our description as the Adams-Bashforth-type integrator (ABTI). 
Prior to this, there have been successful cases where traditional time discretization methods were extended to the complex plane for the construction and solution of differential equations \cite{Corliss1980a,Fornberg2011,Orendt2009,Hansen2009}. Numerical differentiation methods based on interpolation at the unit roots of the complex plane have effectively overcome the ill-conditioning of classical difference schemes \cite{Lyness1967,Fornberg1981}.
\par 
Considering the initial value problem \( \frac{\mathrm{d}}{\mathrm{d}t} u(t) = f(t,u(t)) \) with given initial data \( u(0) = u^0 \), the ABTI can be expressed in matrix form 
\begin{equation}
 \mathbf{u}^{[n+1]}=\mathbf{A}\mathbf{u}^{[n]} + r \mathbf{B}(\alpha)\mathbf{f}^{[n]},  \text{ when } n \ge 1, 
 \label{eq:main_integrator}
\end{equation} 
where \( \mathbf{u}^{[n+1]}, \mathbf{u}^{[n]} \) represent the solution vector, \( \mathbf{A}, \mathbf{B}(\alpha) \) are discrete-time matrix and \( \mathbf{f}^{[n]} = f(\mathbf{t}^{[n]}, \mathbf{u}^{[n]}) \).  Here, the parameter \( \alpha = (t^{n+1} - t^{n})/r \) under the uniform temporal discrete grid. The initial solution vector \( \mathbf{u}^{[0]} \) obtained using so-called iterator \( \mathbf{u}^{[0]} = u^0 \mathbbm{1} + r\mathbf{B}(0)\mathbf{f}^{[0]} \). Finally, the classical numerical solution, represented by the variable \( u^{n+1} \), is reconstructed at discrete time level \( n+1 \) by the component of the solution vector, denoted by \( \mathbf{u}^{[n+1]} \). In the context of a ABTI situation, \( u^{n+1} \) is equal to the arithmetic mean value of all entries in \( \mathbf{u}^{[n+1]} \).
Indeed, if we were to take the arithmetic mean of both sides of equation \eqref{eq:main_integrator}, we would be able to discern the relationship between ABTI and classical ABs. This is due to the fact that the AB method is designed to perform polynomial approximation with previous multiple nodes for the function \( f(t, u(t)) \) in the integral form of the ordinary differential equation (ODE)
\begin{equation}
u(t_{n+1}) = u(t_n) + \int_{t_n}^{t_{n+1}}f(t,u(t)) \mathrm{d}t. \nonumber
\end{equation}
The distinction between the ABTI and ABs can be attributed to the manner of approximation employed for the function \( f(t, u(t)) \). Notwithstanding, there are notable enhancements in the stability of the ABTI for high-order schemes when compared to ABs and other explicit schemes.
\par
In order to gain insight into the stability boundaries of ABTI, it is first necessary to review some of the fundamental concepts associated with the linear stability analysis \cite[Chapter 5]{Hairer1996}. The Dalhquist problem, given by the differential equation \( \frac{\mathrm{d}}{\mathrm{d}t} u(t) = \lambda u(t),  \lambda \in \{z\in \mathbb{C}; \mathtt{Re}(z) \le 0 \} \) with initial data \( u(0) = 1 \), can be discretized using the first-order Adams-Bashforth  (AB1) method. The numerical solution can then be rewritten as  \( u^{n+1} = R(z)u^n\)  where \( R(z) = 1+z \) be called stability function with \( z:= \lambda \tau \),  and \( \mathcal{S} := \{z \in \mathbb{C}; \lvert R(z) \rvert < 1\} \) the region of absolute stability correspondingly. Similarly, these concepts can be extended to encompass ABTI, as defined below.
\begin{definition}
\label{def:absolute_stability_region}
Under the consideration of ABTI for Dalhquist problem, the matrix-value stability function defined by  \( \mathbf{R}(z) = \mathbf{A} + z \mathbf{B}(\alpha)/\alpha \) with \( z:= \tau \lambda \) such that the numerical solution vector \( \mathbf{u}^{[n+1]} = \mathbf{R}(z)\mathbf{u}^{[n]} \) and the absolute stability region described by spectral radius \( \mathcal{S}_{ q } := \{z \in \mathbb{C}; \rho(\mathbf{R}(z)) < 1 \}  \) where \( \mathbf{R}(z) \) is a \( q \)-dimensional square matrix where \( q >1 \).
\end{definition}
Regarding the ABTI algorithm \cite[Section 2.2]{Buvoli2019}, Buvoli gave a guess to quantify the degree of stability improvement compared to existed explicit schemes. For the AB method, we have a general understanding that the region of absolute stability will shrink to origin rapidly as the order of accuracy increases. This means that stronger time-step constraints are required when calculating stiffness problems. At the same time, when the scheme is applied to a parabolic problem, the number of Courant-Friedrichs-Lewy (CFL) conditions is smaller, i.e., strong step size limit is also required. In contrast, using the definition \ref{def:absolute_stability_region} to characterize stability, Buvoli conjecture that the region of absolute stability of the ABTI will tend to a limiting region of absolute stability (centered on the origin and \( \tfrac{1}{e} \) as the left semicircle of radius) as the approximation order increases, viz,  
\begin{equation}
\label{eq:uniform_stability_domain}
\lim\limits_{q\to \infty} \mathcal{S}_{q} \to \mathcal{S}_{\infty} := \left\{z\in \mathbb{C}, \mathtt{Re}(z) < 0 \text{ and } |z|< \frac{1}{e} \right\}, \nonumber
\end{equation} 
where \( e \) is the Euler's number. 
This signifies that the issue of stability anxiety can be effectively addressed through the utilisation of a higher-order explicit scheme. In other words, when a uniform step size constraint is applied, ABTI is still capable of ensuring stability for arbitrary higher-order discrete calculations. Furthermore, given the nature of component parallel computing in ABTI, concerns regarding the computational efficiency of this method are unwarranted, provided that the requisite computing resources are available. To the best of my knowledge, this result has only been numerically illustrated \cite[Figure 2.5]{Buvoli2019}, and has not yet been mathematically rigorously proven. This is one of the tasks that should be completed in this work.
\par
In this paper, we disprove the conjecture put forth by T. Buvoli through a series of methods and provide a criterion for determining the relationship between the parabolic stability radius and the maximum approximation accuracy. On the other hand, the accuracy of original ABTI is one order of magnitude lower than the ideal order. We identify the underlying cause of this discrepancy and restore the accuracy of ABTI through a straightforward correction. By applying the tensor matrix eigenvalue formula, we extend the results of linear stability analysis to linear parabolic equations, identifying the optimal parabolic CFL condition for \( L^2 \)-stability. Finally, a series of illustrative numerical examples are presented.
\par
Throughout this paper, we denote by \( C \) a generic positive constant that is independent of the temporal and spatial size of discretization \( \tau, h \). Let \( \mathbf{B} \) as the abbreviations of \( \mathbf{B}(\alpha)/\alpha \) without causing ambiguity. \( \mathtt{BV}(\Omega) \) means the bounded variation space equided with norm \( \operatorname{TV}( f ):= \sup_{P}\Sigma_{i=1}^{n} \lvert f(x_{i}) - f(x_{i-1}) \rvert \) and \( L^2( \Omega) \) the classical  square integrable space. $\gamma_{n}(z):= z^n/n! $ is denoted as Gelfand-Shilov function, for which the exponential function \( e^z \) can be represented by $ \Sigma_{n = 0}^{\infty}\gamma_n(z)  $ in sense of series expansion. $\mathtt{Re}(z)$ and $\mathtt{Im}(z)$ stand for the real and imaginary part of complex number $z$ respectively. The   $q$ distinct $q$-th roots of unity $\omega_{j}:=e^{\frac{i 2\pi j}{q}}$ for all $j=1,2, \ldots, q$ where the imaginary unit $i = \sqrt{-1}$. The discrete complex time nodes \( t_j^{[n]} := t_{n} + r \omega_{j} \) for all \( j = 1, \ldots, q \) where \( t_{n} \) is traditional time node and \( r  \) is arbitrary non-negative real number stand for radius. Correspondingly, \( u_{j}^{[n]} \) or \( f_{j}^{[n]} \) represent the approximation value or sampling value at the complex time discrete node \( t_j^{[n]} \). 
\par
The rest of this paper is organized as follows. In Section \ref{sec:Methodological review},  we review the construction of the ABTI algorithm and identify a certain structure present in its matrix expression. At the same time, we compare it with the classical first-order Adams-Bashforth method to facilitate our understanding of such algorithms. In Section \ref{sec:The Main Results}, we presents the main results of the paper, including the mathematical analysis of the Buvoli's conjecture and subtle modifications to the original scheme to achieve the optimal convergence order. For the sake of readability, some of the more tedious derivations are placed in Appendix \ref{sec:Derivation of Characteristic Polynomial}. In Section \ref{sec:Applications}, 
we will present the application of the ABTI method to parabolic equations, deriving the \( L^2 \)-stability and \( L^2 \)-error under conditions compatible with ODE problems. In Section \ref{sec:Numerical Verifications} 
we will numerically validate our analytical results.  In Appendix \ref{sec:Derivation of Characteristic Polynomial}, we derives the characteristic polynomial of the amplification matrix, and it will be placed at the end of the paper as preparatory work for Section \ref{sec:The Main Results}.
\section{Methodological review}
\label{sec:Methodological review}
In this section, we restate the formulation of ABTI \cite[Section 2]{Buvoli2019}.
 As a starting, we consider a typical ODE problem
\begin{equation}
\label{eq:ODE}
\tfrac{\mathrm{d}}{\mathrm{d}t} u = f(t,u), \text{ with } u(0) = u^0, 
\end{equation}
and its numerical treatment by ABTI. 
\par 
We outline the main characteristics of the methods by comparing them with the traditional AB1 method, specifically the explicit Euler method. The comparison ranges from the AB1 scheme to higher-order time-stepping methods in the complex plane. The AB1 method requires the initial value at \( t_0 \) and obtains the numerical solution for all \( n \ge 0 \) through repeated iterations from \( t_n \) to \( t_{n+1} \). This process approximates the values at subsequent time steps. Since the time-stepping scheme constructed from \( t_n \) to \( t_{n+1} \) is a linear approximation, the optimal approximation, in the usual sense, can only achieve first-order accuracy. To construct higher-order AB methods, values from previous time nodes are typically needed. It is well-known that derivatives describe the local behavior at a specific moment. However, the involvement of non-local nodes in higher-order numerical schemes and the values they approximate inherently lead to a contradiction, similar to the overfitting phenomenon of higher-order polynomial approximation in data fitting. This contradiction can only be resolved when the time step approaches zero. Consequently, the stability of the scheme requires more stringent constraints on the time step size.  
%
\par 
We treat the time as a complex variable and consider the original time nodes \( t_n \) as vector,  centered at \( t_n \) around \( q \) unit roots, and denote them as \( \mathbf{t}^{[n]} \). The implementation of this algorithm involves three main steps:
\begin{enumerate}
\item
\textbf{Preprocessing} (Initialization Vector) – In the traditional AB method, the initial value \( u^{0} \) is strictly equal to the initial condition of the equation. In the ABTI method, the initialization value is essentially a vector, but we only know the value at the real time node \( t_0 \). Therefore, we need to somehow establish a connection between \( u^{0} \) and \( \mathbf{u}^{[0]} \), which can be done using the information provided by the differential equation. This is exactly what we will introduce next: the iterators.
\item
\textbf{Time-stepping} (Solution Vector Updating)– From the constructed scheme, we compute the numerical solution vector \( \mathbf{u}^{[n+1]} \) from the given \( \mathbf{u}^{[n]} \) and the right-hand-side function \( f \). This is exactly what we will introduce next: the propagator.
\item
\textbf{Postprocessing} (Numerical Solution Reconstruction) – From the previous two steps \textbf{preprocessing} and \textbf{time-stepping}, we obtain the solution vectors \( \mathbf{u}^{[n]} \) for all \( n \ge 0 \). However, the primary goal is to obtain the values at the real time node \( t_{n} \). The information at \(  \mathbf{t}^{[n]} \) is then used to reconstruct the approximate value at the real time node \( t_n \). According to Schwarz's reflection principle, we know that the function values of an analytic function at conjugate nodes are complex conjugates of each other. Thus, the numerical solution at the real time node \( t_n \) can be expressed as the arithmetic mean of all components of the solution vector \( \mathbf{u}^{[n]} \), namely \( u^{n}=\frac{1}{q}\sum_{j=1}^{q}u_{j}^{[n]}\).
\end{enumerate}
\par 
Overall, from the perspective of the \textbf{time stepping} process, ABTI has a high similarity to the AB1 method and, as \( r \to 0 \), ABTI can degenerate into the AB1. The locality of the ABTI method aligns with the locality of the derivative, and this consistency ensures that the ABTI method maintains good numerical stability. In other words, we can also consider preprocessing and postprocessing as processes of dimension extension and dimension reduction, respectively.
\par
 Next, we will review the details of the scheme construction. Since the Cauchy–Kowalevski theorem guarantees analyticity of solution of original ODE \eqref{eq:ODE} inside the proper circular domain, the \( u(t_j^{[n+1]}) = u(t_{n+1}+r \omega_j) = u(t_{n}+\tau+r \omega_{j}) \) can be approximated by the Laurent expansion  at \( t_{n+1} \) or \( t_{n} \) that 
\begin{subequations}
\begin{align}
 u(t_j^{[n+1]})  \approx u^{n+1} + \sum\limits_{\nu=1}^{q}  \frac{u^{(\nu)}(t_{n+1})}{\nu!} (r \omega_{j})^{\nu}, \label{eq:pre_iterator}\\
 u(t_j^{[n+1]})  \approx u^{n} + \sum\limits_{\nu=1}^{q}  \frac{u^{(\nu)}(t_{n})}{\nu!} (r \omega_{j}+\tau)^{\nu}. \label{eq:pre_propagator}
\end{align}
Due to  the similarity of treatment about expansion coefficients in \eqref{eq:pre_iterator} and \eqref{eq:pre_propagator}, the statement focus on the second one. The approximation \eqref{eq:pre_propagator} whose coefficients of the Laurent series can be represented by the Cauchy integral formula \cite{Bornemann2010}
\begin{align*}
\displaystyle
\frac{u^{(\nu)}(t_n)}{\nu !} = 
\begin{cases}
\dfrac{1}{2 \pi i}\oint_{\Gamma}\dfrac{u(z)}{(z-t_{n})}\mathrm{d}z, & \nu = 0, \\
\dfrac{1}{2 \pi \nu i} \oint_{\Gamma}\dfrac{u^{\prime}(z)}{(z-t_n)^{\nu}}\mathrm{d}z = \dfrac{1}{2 \pi \nu i} \oint_{\Gamma}\dfrac{f(z, u(z))}{(z-t_n)^{\nu}}\mathrm{d}z , & \nu \ge 1.
\end{cases}
\end{align*}
Equivalently, choosing a relatively simple \( \Gamma \) to be a circular of contour radius \( r \) centered at \( t_n \) and utilizing the change of variables  \( z = t_n + re^{i \theta} \), then the above expression is equivalent to 
\begin{align*}
\frac{u^{(\nu)}(t_n)}{\nu !} = 
\begin{cases}
\dfrac{1}{2\pi} \int_{0}^{2\pi}u(r e^{i \theta} + t_n) \mathrm{d}\theta, & \nu = 0, \\
\dfrac{1}{2 \pi \nu r^{\nu-1}}\int_{0}^{2\pi}\dfrac{f(re^{i \theta} + t_n, u(re^{i \theta} + t_n))}{e^{i(\nu-1)\theta}}\mathrm{d}\theta,  & \nu \ge 1.
\end{cases}
\end{align*}
\end{subequations}
The processing of integrals in the continuous sense remains challenging for computer programs and necessitates additional discretization. The trapezoidal rule \cite{Trefethen2014}
 discrete the above formula, we consequently have that
\begin{equation}
\label{eq:trapezoidal_rule}
\frac{u^{(\nu)}(t_n)}{\nu !} \approx
\begin{cases}
\frac{1}{s}\sum\limits_{k=1}^{s}u_k^{[n]}, & \nu = 0, \\
\frac{1}{\nu r^{\nu-1}}\cdot \frac{1}{s}\sum\limits_{k=1}^{s}f_k^{[n]}e^{-i(\nu-1)\theta_k} = \frac{\hat{f}_{\nu-1}^{[n]}}{\nu r^{\nu-1}}, & q \ge \nu \ge 1,
\end{cases}
\end{equation}
where \( u_k^{[n]}, f_k^{[n]} \) stand for the value of functions \( u, f \) at sample point \( t_n + r\omega_k \) and $ s $ means number of sampling points. Our study shows that the choice of \( s \) leads to the following scenarios and corresponding outcomes:
\begin{align*}
s
\begin{cases}
< q, & \text{the scheme is unstable and achieves no computational accuracy}; \\
= q, & \text{the scheme is conditionally stable with \( (q-1) \)-th order accuracy}; \\
\geq q+1, & \text{the scheme is conditionally stable with \( q \)-th order accuracy}.
\end{cases} 
\end{align*}
Although choosing \( s = q+1 \) may seemingly improve the accuracy of the final approximation scheme, the computational cost does not differ significantly from directly increasing \( q \). When \( s > q+1 \), the scheme can only maintain \( q \)-th order accuracy, while its computational cost remains proportional to the value of \( s \).  This paper focuses exclusively on the case \( s = q \) for this reason.
Thence, \eqref{eq:pre_iterator} and \eqref{eq:pre_propagator}, we have 
\begin{subequations}
\begin{align}
u(t_j^{[n+1]}) \approx & u^{n+1} + r \sum\limits_{\nu=1}^{q}\frac{\omega_{j}^{\nu}}{\nu}\hat{f}_{\nu-1}^{[n+1]} \nonumber \\ 
=  & \frac{1}{q}\sum\limits_{k=1}^{q}u_k^{[n+1]} + r \sum_{\nu=1}^{q}\int_{s=0}^{\omega_{j}}s^{\nu-1} \mathrm{d}s \cdot  \hat{f}_{\nu-1}^{[n+1]} =:  u_j^{[n+1]}, \label{eq:component_iterator} \\
\text{ and   ~~~~~~~~~~~~~~} &  \nonumber\\
u(t_j^{[n+1]}) \approx & \frac{1}{q}\sum\limits_{k=1}^{q}u_k^{[n]} + \sum_{\nu=1}^{q} \left(\frac{\hat{f}_{\nu-1}^{[n]}}{\nu r^{\nu-1}}\right) (r \omega_{j} + \tau)^{\nu} \nonumber \\
 \xlongequal{\tau  = \alpha r } & 
\frac{1}{q}\sum\limits_{k=1}^{q}u_k^{[n]} + r \sum\limits_{\nu=1}^{q}\frac{(\alpha+\omega_{j})^{\nu}}{\nu}\hat{f}_{\nu-1}^{[n]} \nonumber \\
= & \frac{1}{q}\sum\limits_{k=1}^{q}u_k^{[n]} + r \sum\limits_{\nu=1}^{q}\int_{s=0}^{\alpha+\omega_j}s^{\nu-1}\mathrm{d}s \cdot \hat{f}_{\nu-1}^{[n]}  =: u_j^{[n+1]}, \label{eq:component_propagator}
\end{align}
\end{subequations}
where \( 1 \le j \le q \).
\begin{figure}[!t]
\centering
\includegraphics[scale=.3]{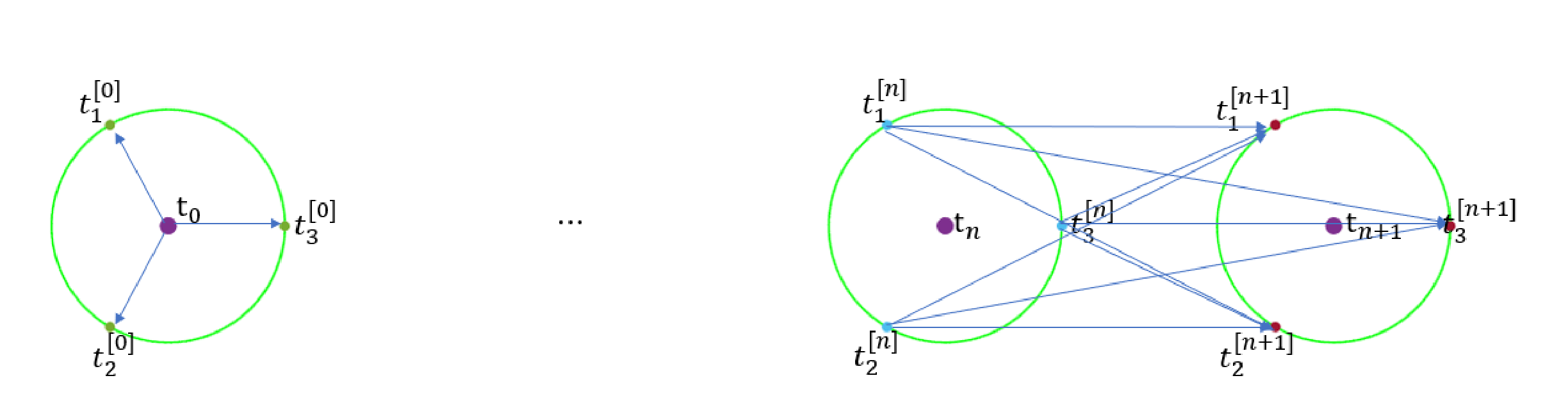}
\caption{Obtain the initial solution vector $\mathbf{u}^{[0]}$ by iterator \eqref{eq:iterator}(Left) and update each of component \( u_j^{[n+1]} \) of solution vector \( \mathbf{u}^{[n+1]} \) by propagator \eqref{eq:propagator} (Right). }
\label{fig:fig_dci}
\end{figure}
Summarily, the component form of time discrete scheme \eqref{eq:component_iterator} and \eqref{eq:component_propagator} can be represented by matrix-vector form
\begin{subequations}\label{eq:time scheme}
\begin{align}
\text{iterator: }  &  \mathbf{u}^{[n+1]} = \mathbf{A}\mathbf{u}^{[n+1]} + r\mathbf{B}(0)\mathbf{f}^{[n+1]}, \label{eq:iterator}\\
\text{propagator: }  &  \mathbf{u}^{[n+1]}=\mathbf{A}\mathbf{u}^{[n]} + r \mathbf{B}(\alpha)\mathbf{f}^{[n]},  \text{ when } n \ge 1, \label{eq:propagator}
\end{align} 
\end{subequations}
where 
\begin{itemize}
\item
\( \mathbf{u}^{[n]}:=[u_j^{[n]}]_{j=1}^q \) and \( \mathbf{f}^{[n]}:=[f_j^{[n]}]_{j=1}^q \) stand for solution vector and right-hand side vector respectively;
\item 
\(  \mathbf{A} \)  is an \( q \)-dimensional all-one square matrix multiplied by \( \frac{1}{q} \);
\item $\mathbf{B}(\alpha) := \mathbf{S}(\alpha)\mathbf{F}$ whose
\begin{align*}
\mathbf{S}(\alpha) =
\begin{bmatrix}
\sigma_{1,1} & \ldots & \sigma_{1,q} \\
\vdots & \vdots & \vdots \\
\sigma_{q,1} & \ldots & \sigma_{q,q} 
\end{bmatrix} 
\text{ and }
\mathbf{F} = 
\frac{1}{q} 
\begin{bmatrix}
\omega_1^0 & \ldots & \omega_q^0 \\
\vdots & \ddots & \vdots \\
\omega_{1}^{1-q} & \ldots & \omega_q^{1-q} \\
\end{bmatrix}
\end{align*}
where \( \sigma_{j,k} = \sigma_{j,k}(\alpha) :=  \int_{x=0}^{\alpha + \omega_j}x^{k-1} \mathrm{d}x\). Here, \( \mathbf{S}(\alpha) \) can be regarded as a Vandermonde matrix in the sense of integration and \( \sqrt{q}\cdot\mathbf{F} \) is a  discrete Fourier matrix.
\end{itemize}
One of the main utilities of the iterator is to obtain the initial solution vector \( \mathbf{u}^{[0]} \) by replacing \( n+1 \) as \( 0 \) directly, namely
\begin{equation}
\mathbf{u}^{[0]} = u^0 \mathbbm{1} + r\mathbf{B}(0)\mathbf{f}^{[0]}
\end{equation}
where \( \mathbbm{1} \) is \( q \)-dimensional all-one column vector. 
\par 
In the original paper \cite{Buvoli2018}, the number of sampling points and the number of terms in the Taylor expansion  are equal. 
Here, we find that the accuracy of the scheme is one order lower that the ideal approximation of order \( q \). The \( q \)-order accuracy is referred to as the ideal approximation because the construction of the scheme involves both the truncation error of the Taylor expansion and the error in approximating the Cauchy integral formula, with the latter achieving exponential convergence, thus not dominating the overall error. Throughout the article, we name the ABTI method with $q$-order accuracy as ABTI$_{q}$ and the number of corresponding sampling points should be $q+1$, that is, the number of items Taylor expands should be $q+1$ when constructing various types.
\par
In addition to its good stability, ABTI also possesses several other features. From equation \eqref{eq:propagator}, it is not difficult to observe that once the matrix is determined, the time-stepping procedure is very similar to the explicit Euler method, making the programming of ABTI very straightforward. Since it is an explicit scheme, we can see that each component of the unknown vector on the left-hand side of \eqref{eq:propagator} can be computed independently. This feature becomes even more evident when the matrix form of the scheme reverts to the component form as presented in the original paper. Upon examining the matrix representation of the scheme, we observe that the matrix exhibits a certain structure, which also opens up possibilities for algorithm acceleration or GPU programming. For example, there are many acceleration algorithms for the multiplication of Fourier matrices and vectors e.g. fast Fourier transform (FFT)\cite{Cooley1965}.
\section{The Main Results}
\label{sec:The Main Results}
This section includes two parts. First, there is the theoretical analysis regarding the conjecture that ABTI possesses limiting stability. Secondly, an analysis of the reasons why the original ABTI cannot achieve the desired convergence order is provided, along with potential solutions.
\subsection{Disproval  of a conjecture}
At the outset of this subsection, we will present a concise overview of the proof methodology that we have devised to evaluate Buvoli's conjecture. A relatively straightforward approach is to identify the characteristic polynomial of the amplification matrix and examine the relationship between the control parameter $ z \in \mathbb{C} $ and the Schur stability of the polynomial. This involves locating $z$ the complex plane in a way that ensures all the zeros of the polynomial are situated within the unit disk. In the event that the amplification matrix is expressed explicitly, the characteristic polynomial is not readily apparent. Initially, it is determined that the matrix $ \mathbf{B}(\alpha) $ in the scheme can be expressed as the product of two matrices \( \mathbf{B}(\alpha) = \mathbf{S}(\alpha)\mathbf{F} \) with a particular structure, using the aforementioned matrix structure and the properties of the block matrix determinant, it was determined that a particularly simple band matrix exhibits an identical characteristic polynomial to that of the matrix $ \mathbf{A} + z \mathbf{B}(\alpha) $. The difference equation is established with the value of the determinant, and the explicit expression the characteristic polynomial is obtained through two algebraic matrix inversions. Subsequently, the analysis of the absolute stability region is transformed into an analysis of the zeros of the polynomial. 
\par 
By employing the analytical approach of \textit{the root locus curve} \cite[V.I]{Hairer1996}, a straightforward variable substitution is employed to derive \textit{the parabolic radius of the stable region}\footnote{Here, the parabolic radius refers to the length of the segment along the real axis, starting from the origin and extending toward \( - \infty \), that reaches the boundary of the stability region. Essentially, this means that the spectral points of the Laplace operator \( \Delta \) in the parabolic equation lie on the negative real axis and can extend to \( - \infty \). The discretization of the Laplace operator, or more specifically, the spectrum of the discretized Laplace operator, must fall within the allowable range of the time scheme's parabolic radius to ensure the stability of the scheme.}, which is a crucial prerequisite for validating the conjecture. Consequently, our focus can be narrowed to the distribution of zeros of a real-coefficient polynomial with solely real zeros. These polynomials can be regarded as a polynomial sequence with respect to the index $q$, and the generating function of this polynomial sequence is obtained by simple derivation. Subsequent value of the polynomial at a specified real point is represented as the Fourier transform value of a function derived from a generating function through the Cauchy integral formula and variable substitution. Ultimately, in consideration of the relationship between the function's smoothness and the decay of the Fourier coefficient, it is demonstrated that the conjecture cannot be guaranteed to hold when $q$ tends to infinity. 
\par 
Nevertheless, this scheme demonstrates superior stability compared to that of ABs of equivalent accuracy, particularly in the case of high-order approximation. The stability of the scheme can be expressed as follows: as the approximation accuracy increases, the stable region will gradually diminish, and the rate of this diminution will become increasingly slow. Consequently, we also provide the quantization of the maximum  approximation accuracy that can guarantee the stability of the scheme for a given parabolic radius. In other words, this is a necessary condition for the stability of the scheme for arbitrary approximation accuracy when solving parabolic problems. Conversely, for a given precision, the parabolic radius can be determined by solving for the smallest modulus zero of the polynomial.
\par
To highlight the core ideas of the proof and enhance the readability of the article, we have placed the derivation of the characteristic polynomial of \(  \mathbf{A} + z \mathbf{B} \) in the appendix \ref{sec:Derivation of Characteristic Polynomial}. In this section, we only present the final expression of the characteristic polynomial
\begin{equation}
\label{eq:characteristic polynomial}
 p_q(\lambda; \mathbf{A} + z \mathbf{B}) = f_q( \lambda; z) + f _{q-1}( \lambda; z) - \gamma_{q}(-z/\alpha), 
\end{equation} 
where \( f_q( \lambda; z) :=  \sum_{j=0}^{q} \gamma_{q-j}((j+1)z)(-\lambda)^{j} \).

\par
From the Figure \ref{fig:zeta_z_plane}, we can observe that the stability region of ABTI is the common area enclosed by a series of \textit{roots locus curve} of the polynomial \( p_q( z; \theta) := p_q(e^{i \theta}; \mathbf{A} + z \mathbf{B}) \) with respect to \( z = \lambda \tau \in \mathbb{C} \), parameterized by \( \theta \in [-\pi, \pi) \). These \textit{roots locus curve} can be divided into two types: one is a pseudo-semicircle \( \Gamma^{(1)} \) passing through the origin, and the other consists of multiple quasi-concentric circles \( \Gamma^{(2)}_j \) that do not pass through the origin. We can order the set \( \{\Gamma^{(2)}_j \}_{j \ge 1} \) by their distance from the origin, such that \( \Gamma^{(2)}_1 \) is the quasi-circle closest to the origin that does not pass through it. 
\par
For convenience in handling and observation, we divide the original polynomial by a factor \( (-e^{i \theta})^q \), which does not change the zeros of the original polynomial, and introduce the variable \( \zeta = - e^{-i \theta} z \). This results in the following variant polynomial,
\begin{equation}
\tilde{p}_{q}(\zeta; \theta) := \sum_{j=0}^{q} \gamma_{q-j}((j+1) \zeta) - e^{-i \theta}\sum_{j=0}^{q-1} \gamma_{q-1-j}((j+1)\zeta) - \gamma_{q}(-\zeta/\alpha), 
\end{equation}
whose zeros have the same modulus as those of \( p_{q}(z; \theta) \). Without loss of generality, we set the parameter $\alpha=1$ throughout this paper unless otherwise stated.
\par 
{The existence of a limiting parabolic radius is a necessary condition for the conjecture to hold. Considering the parabolic radius is not merely to restrict the characteristic polynomial to real coefficients, but rather because for certain polynomials \( p \) with \( q \geq 1 \), the \textit{roots locus curve} \( \Gamma^{(1)} \) may not necessarily include the imaginary axis. These pseudo-semicircles, which do not contain the imaginary axis, will slightly deviate away from the imaginary axis as they approach it. This characteristic limits the ability of the ABTI scheme to handle hyperbolic problems and provides one of the reasons for considering the parabolic radius instead of the hyperbolic radius. If we construct a circle \( \Gamma^* \) with the origin as the center and the parabolic radius as the radius, it represents the largest semicircular region that can be contained within the stability region. This property requires examination through the deformed polynomial. Notably, there is a one-to-one correspondence between \( \Gamma^{(2)}_1 \) and \( \gamma^{(2)}_1 \), where \( \gamma^{(2)}_1 \) is a convex small circle on \( \zeta \)-plane with similar definition of \( \Gamma^{(1)} \) or \( \Gamma^{(2)}_{j} \). Therefore, we know that the distance between \( \Gamma^{(2)}_{1} \) and the origin increases gradually from \( \theta = \pm \pi \) to \( \theta = 0 \). In other words, the region enclosed by \( \Gamma^* \) can be completely contained within the region enclosed by \( \Gamma^{(2)}_1 \). Through experiments, we further observe that the radius of \( \gamma^{(2)}_1 \) decreases as \( q \) increases, causing \( \Gamma^{(2)}_1 \) to approach a more regular circular shape.}
Another reason is that considering the equivalent polynomial allows for a more concise analysis and expression.

\begin{figure}[!h]
\centering
\includegraphics[scale=.3]{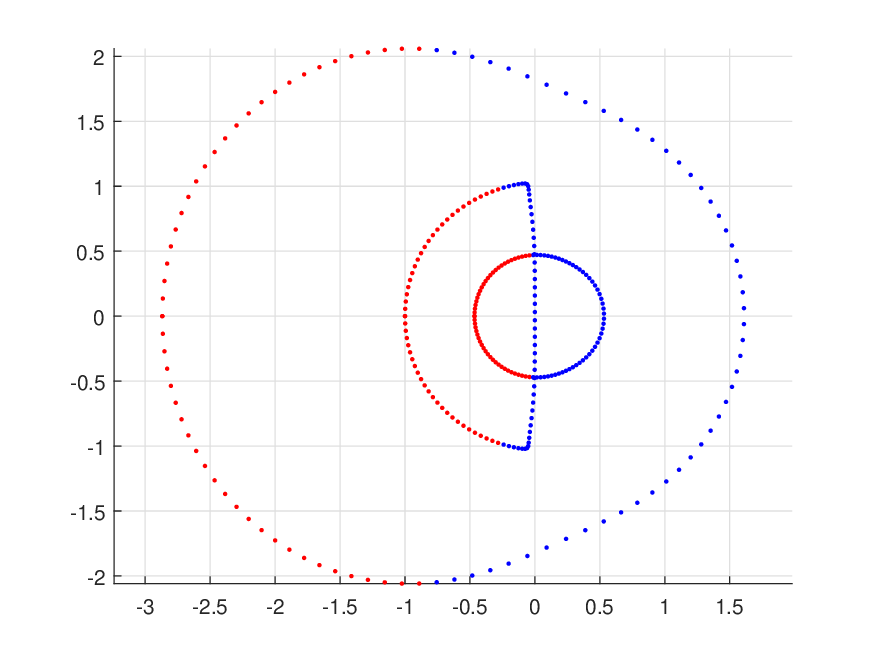}
\includegraphics[scale=.3]{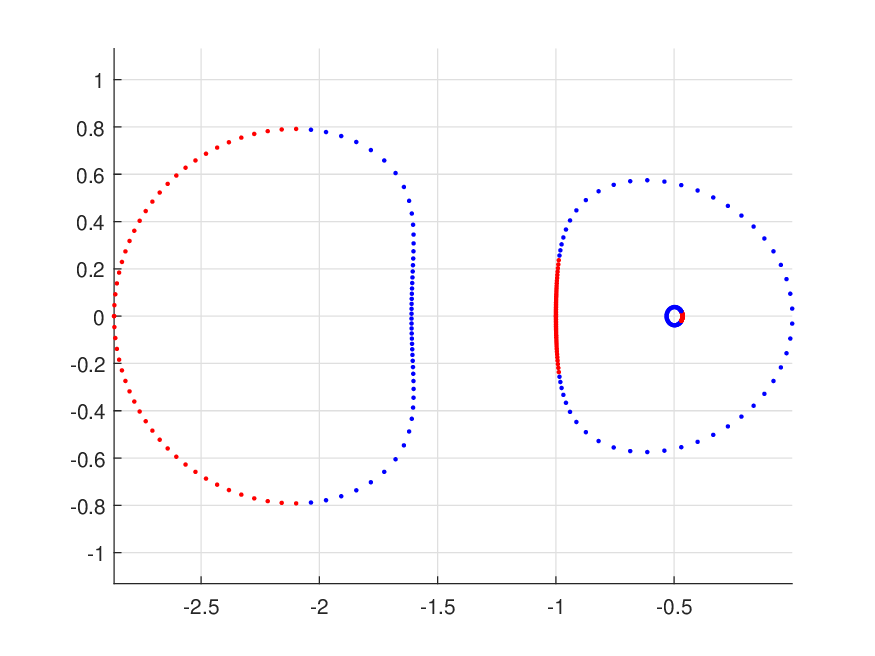}
\caption{The zeros of \( p_{4}(z; \theta) \) on the \( z \)-plane (Left) and  the simple transformation \( \zeta(\theta) = - e^{-i \theta} z(\theta)  \) on the \( \zeta \)-plane (Right). The zeros $\zeta(\theta)$ and $z(\theta)$ be marked by blue whenever \( \theta \in [-\tfrac{\pi}{2}, \tfrac{\pi}{2}] \).}
\label{fig:zeta_z_plane}
\end{figure}
\begin{lemma}
\label{lem:real zeros}
All zeros of  polynomial \( \tilde{p}_n( \zeta ; \theta ) \) lie in the real axis when \( \theta= 0 \) or \( \pi \). Moreover, the parabolic radius used in the linear stability analysis of ABTI is given by the minimum modulus of the zeros of \( \tilde{p}_n( \zeta ; \pi ) \).
\end{lemma}
\begin{proof}
From the expression of polynomial \( \tilde{p}_n(\zeta; \theta) \), it is easy to check that  
\begin{equation}
\tilde{p}_n(\zeta; \theta) -  \tilde{p}_{n}(\bar{\zeta}; - \theta) = 2 \mathtt{Im}(\tilde{p}_n(\zeta; \theta)). \nonumber
\end{equation}
Here, we utilize the fact that \( \tilde{p}_n(\zeta; \theta) \) and \( \tilde{p}_{n}(\bar{\zeta}; - \theta) \) form a pair of complex conjugates. For arbitrarily fixed \( \theta \in [- \pi, \pi] \), if \( \zeta_0 \) is the zero of the polynomial \( \tilde{p}_n(\zeta; \theta) \), we can lead  \( \tilde{p}_{n}(\bar{\zeta}_{0}; - \theta)  = 0 \) from the fact \( \tilde{p}_n(\zeta_0; \theta) =  2 \mathtt{Im}(\tilde{p}_n(\zeta_0; \theta)) = 0 \)  in preceding formula. This implies that the zeros of polynomial \( \tilde{p}_n(\zeta; \theta) \) must be zeros of \( \tilde{p}_{n}(\bar{\zeta}; - \theta)  \). In other word, the distribution of  the zeros associated with \( \theta \) are strictly symmetric about real axis in the Figure \ref{fig:zeta_z_plane} from the continuous dependence of zero distribution of polynomials on coefficients. We demonstrate the case for \( \theta = 0 \) as an example. The two families of zeros of \( \tilde{p}_{n}(\bar{\zeta}; \varepsilon ) \) and \( \tilde{p}_{n}(\bar{\zeta}; - \varepsilon) \) are symmetrically distributed on both sides of the real axis for any \( \varepsilon > 0 \). As \( \varepsilon \to 0 \), the zeros of \( \tilde{p}_{n}(\bar{\zeta}; \pm \varepsilon ) \) continuously approach their axis of symmetry, meaning that the zeros of \( \tilde{p}_{n}(\bar{\zeta}; 0) \) lie on the real axis. Similarly, we can also show that for \( \theta = \pi \), the zeros of \( \tilde{p}_{n}(\bar{\zeta}; \pi) \) lie on the real axis. Therefore, we can conclude this lemma. 
\end{proof}
\begin{remark}
This lemma requires us to rely on two key observations from the right panel of Fig. 2:
\begin{enumerate}[(1)]
\item 
Each closed curve in the figure represents a certain zero family $\zeta(\theta)$ of the polynomial $\tilde{p}_n(\zeta; \theta)$. Based on extensive experimental observations, the number of these closed curves exactly equals the degree of the polynomial, and the curves never intersect. According to the Fundamental Theorem of Algebra, this rules out the possibility that the real-coefficient polynomial $\tilde{p}_n(\zeta; \pi)$ has a pair of conjugate zeros, as otherwise the number of solutions would not match the degree of the polynomial.
\item 
When plotting these closed curves, we observe that as $\theta$ increases from $-\pi$ to $\pi$, the curves evolve along a consistent direction.
\end{enumerate}
Here, special attention must be paid: when $\alpha = 1$ and $q$ is even, the highest-degree term in $\sum_{j=0}^{q} \gamma_{q-j}((j+1) \zeta)$ is exactly canceled by $\gamma_{q}(-\zeta/\alpha)$. In this case, $\tilde{p}_n(\zeta; \theta)$ becomes a polynomial of degree $q-1$ only. The discrete appearance of the closed curves in the figure is due to the fact that the plot is generated using a finite set of discrete points $\theta$, obtained by uniformly partitioning the interval $[-\pi, \pi]$.
\end{remark} 
\begin{remark}
From another rigorous deduction, we can also judge that \( \tilde{p}_n(\zeta; \theta = 0 \text{ or } \pi) \) is the log-concave polynomial \cite{Huh2022} which has only real zeros, for which the polynomial coefficient are a unimodal sequences \cite{Liu2007}. This kind of polynomial is widely used in combinatorial mathematics.  
\end{remark}
When \( \theta = 0 \), the zeros of the polynomial are non-negative. When \( \theta = \pi \), the zeros of the polynomial are strictly negative. At this point, the absolute value of the zero with the smallest modulus is exactly the parabolic radius of the stability region of the ABTI method. This is the content that the following theorem needs to examine.
\begin{theorem}\label{lem:disprove}
For all \( 1 \le n < \infty \), the zeros of the polynomial \( \tilde{p}_n( \zeta; \pi) \) are real and distinct. Furthermore, for any finite \( N \in \mathbb{N} \), there exists a non-negative \( r_{N} \) such that the interval $( - r_{N}, 0]$ contain no zeros of \( \tilde{p}_{n}(\zeta; \pi) \) for all \( n < N \). In particular, \( r_{N} \) will slowly decrease with the increase of \( N \). 
\end{theorem}
\begin{proof}
\par
In linear stability analysis, when examining the parabolic radius of a scheme, the domain of \( \zeta \) can be restricted from the entire complex plane to the real axis. Lemma \ref{lem:real zeros} specifies the corresponding values of \( \theta \) in this scenario. 
\par
Denote the polynomial sequences \( \tilde{f}_n( \zeta ) := \sum_{j=0}^{n}\gamma_{n-j}((j+1) \zeta) \) and its generating function denoted by
\begin{align*}
\widetilde{ \mathsf{F}}(\zeta,t)  & := \sum_{n=0}^{\infty}\tilde{f}_{n}(\zeta)t^{n} = \sum_{n=0}^{\infty}\left(\sum_{j=0}^{n}\frac{((j+1)\zeta)^{n-j}}{(n-j)!}\right)t^{n} \\
= & \sum_{j=0}^{\infty}\sum_{n=j}^{\infty}\frac{((j+1)\zeta)^{n-j}}{(n-j)!}t^{n} =  \sum_{j=0}^{\infty}\sum_{n=0}^{\infty}\frac{((j+1)\zeta)^{n}}{n!}t^{n+j}\\
 = & \sum_{j=0}^{\infty}t^{j} e^{(j+1)\zeta t} = e^{\zeta t}\sum_{j=0}^{\infty}(t e^{\zeta t})^{j} = \frac{1}{e^{- \zeta t} - t} \text{ with \( | t e^{\zeta t} | < 1 \).}
\end{align*}
Where we interchange the order of summation and recalling the range of index \( n \).
Thence, 
\begin{equation}
\widetilde{ \mathsf{P} }(t; \zeta)  = \frac{ 1 + t }{e^{-\zeta t}-t} - e^{-\zeta t} \nonumber
\end{equation}
is the generating function of \( \tilde{p}_{n}(\zeta; \pi) \) which can be obtained by similar way. 
For \( \zeta \in \mathbb{R}, t \in \mathbb{C} \), the condition \( \zeta < 0, \lvert t \rvert \le 1 \) leads to  \(  | t e^{\zeta t} | < 1\) sufficiently. In fact, by expressing $|t| \le 1$ in polar coordinates as $t = \rho e^{i \theta}$, where $0 \le \rho \le 1$ and $\theta \in [0, 2\pi)$, we obtain $|t e^{\zeta t}| = \rho \cdot e^{\zeta \rho \cos(\theta)}$. At this point, we observe that $e^{\zeta \rho \cos(\theta)} < 1$ together with $\rho \le 1$ directly implies $\rho \cdot e^{\zeta \rho \cos(\theta)} < 1$.  In this way, the value range of variables \( \zeta, t \) in the equation \eqref{eq:to_Fourier_coef} can ensure the convergence of the series when the generating function is derived.
In order to confirm the range of the \( \tilde{p}_{n}(\zeta) \), we turn it to a form of Fourier transformation. 
\begin{equation}
\label{eq:to_Fourier_coef}
\tilde{p}_{n}(\zeta) := \tilde{p}_{n}(\zeta; \pi) = \frac{1}{2 \pi i} \oint_{\left|t\right|=1} \frac{\widetilde{\mathsf{P} }(t; \zeta)}{t^{n+1}}  \mathrm{d}t = \frac{1}{2 \pi}\int_{0}^{2 \pi}\tilde{\mathcal{P}}(\varphi; \zeta)\cdot e^{-i  n \varphi} \mathrm{d}\varphi ,
\end{equation} 
where \( \tilde{\mathcal{P}}(\varphi; \zeta) := \widetilde{ \mathsf{P} }( e^{i \varphi}; \zeta)  \) is a function with period \( 2 \pi \) about \( \varphi \in \mathbb{R} \).
Thus, we have a bridge between the value of the sequences \( \tilde{p}_{n}(\zeta) \) and the spectrum of the function \( \tilde{\mathcal{P}}(\varphi; \zeta) \). 
\par  
If Buvoli's conjecture is valid, then \( \tilde{p}_{n \to \infty }(\zeta) > 0 \) for all \( \zeta \in (-1/e,0) \). However, we can disprove this result from the view of harmonic analysis. 
We notice that  \( \tilde{{\mathcal{P}}}(\varphi; \zeta) \in \mathtt{BV}(\mathbb{R}\times \mathbb{R}) \) which also can be observed from the Figure~\ref{fig:disprove}  and the Figure~\ref{fig:REAL AND IMAGE OF P} directly.  
This can also be proved based on the fact that a complex-valued function is of bounded variation if and only if its real and imaginary parts are both of bounded variation. 
\cite[Table 3.1]{Grafakos2014} has shown the relation between the decay rate of Fourier coefficients of a function \( \tilde{\mathcal{P}}(\varphi; \zeta) \) associated with variable \( \varphi \) and whose smooth properties, viz, $\tilde{p}_{n} ( \zeta )  = \mathcal{O}(n ^{-1})$ since \( \tilde{\mathcal{P}}(\varphi; \zeta) \in \mathtt{BV}(\mathbb{R}\times \mathbb{R}) \)  for all \( \zeta \in (-1/e, 0) \).  Therefore, for arbitrary \( \zeta \in (-1/e, 0) \), \( \tilde{p}_{n \to \infty }(\zeta) \to 0 \). 
\par 
Since the continuity of polynomial \( \tilde{p}_{n}(\zeta) \) and it across the common point \( (0,2) \)  for all \( n \ge 1 \), there are \( N \)  for all \( r_{N} \in (-1, 0) \) such that \( \tilde{p}_{n}(\zeta) > 0 \) whenever \( n \le N \).
\end{proof}
\begin{remark}
Even so, this method has significant advantages compared to traditional AB methods of order 4 and above. By presetting an appropriate parabolic radius, we can determine the maximum permissible accuracy by calculating the integral expression on the right side of the equation \eqref{eq:to_Fourier_coef}. Taking Buvoli's conjectured radius \( 1/e \) as an example, using Trefethen's numerical integration toolbox \footnote{\url{https://www.chebfun.org}} and setting the tolerance of integrand more than \(  10 ^{-12} \), we can determine that the maximum permissible accuracy is $n = 30$. When the radius equals to AB4, namely \( r_n = 0.3 \), the maximum permissible accuracy is $n = 56$.
\begin{table}[h]
\caption{\large The  maximum permissible accuracy for given parabolic radius.}
\begin{tabular}{cccccc} 
\hline\noalign{\smallskip}
\( r =  \) & \( 0.6 \) & \( 0.5 \) & \( 0.4 \) & \( \mathbf{1/e} \) & \( 0.3 \) \\
\noalign{\smallskip}\hline\noalign{\smallskip}
\( n \le  \) & 1 & 2 & 6 & 30 & 56  \\ 
\noalign{\smallskip}\hline
\end{tabular}
\end{table}
\begin{figure}[!h]
	\centering
	\includegraphics[scale=.3]{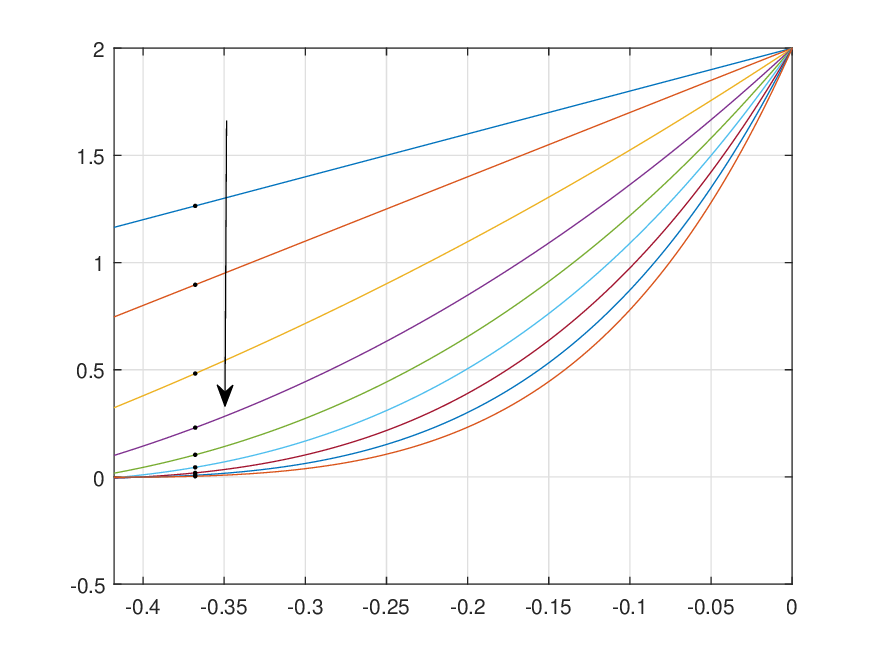}
	\includegraphics[scale=.3]{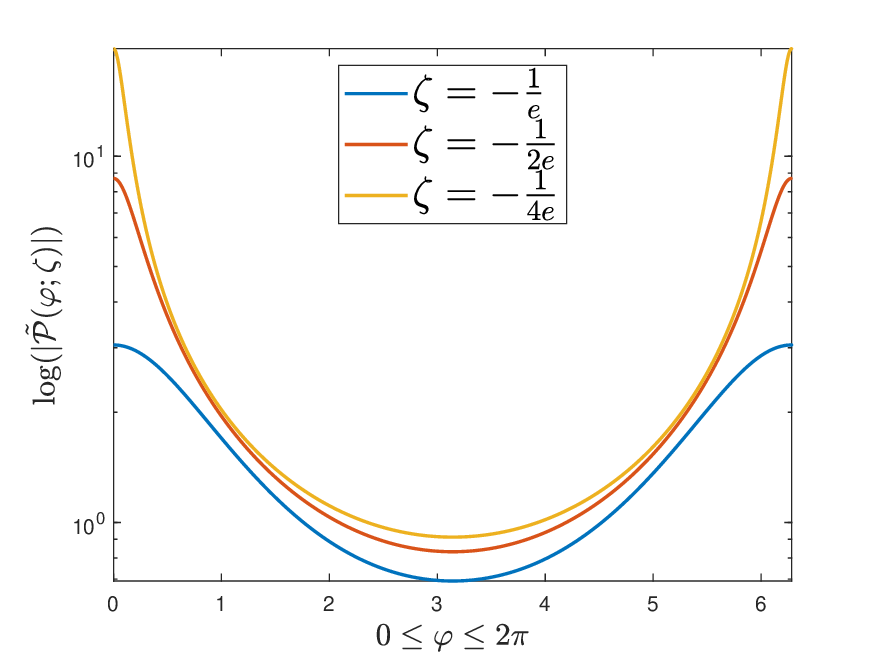}
	 \caption{The local plot of polynomial sequence \( \tilde{p}_{n}(\zeta) \) (Left) and the graph of \( \lvert \tilde{\mathcal{P}}(\varphi; \zeta) \rvert \) with \( \varphi \in [0, 2 \pi] \) and \( \zeta = - \tfrac{1}{e}, - \tfrac{1}{2e}, - \tfrac{1}{4e} \)(Right)}
	\label{fig:disprove}
\end{figure} 
\begin{figure}[!h]
	\centering
	\includegraphics[scale=.3]{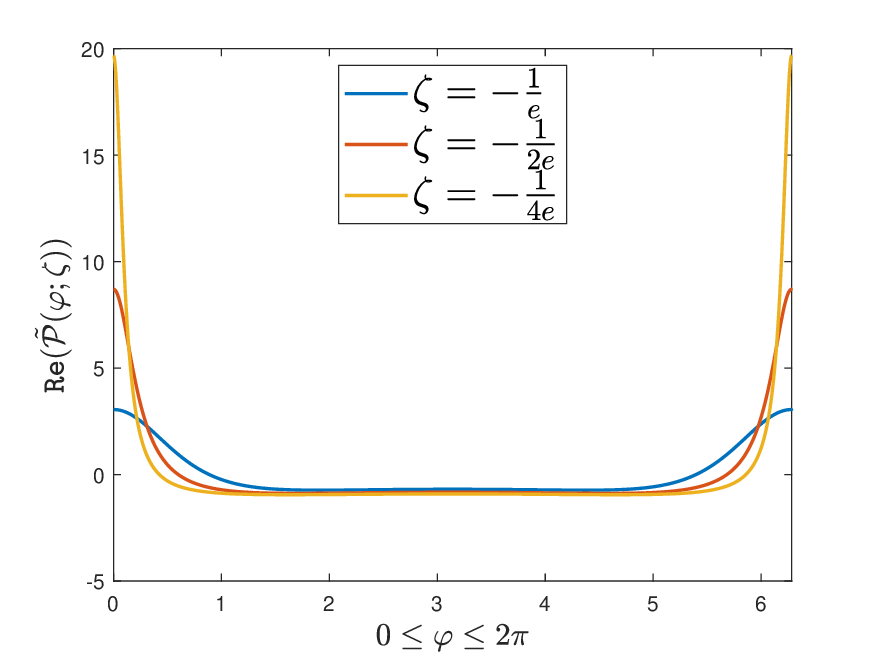}
	\includegraphics[scale=.3]{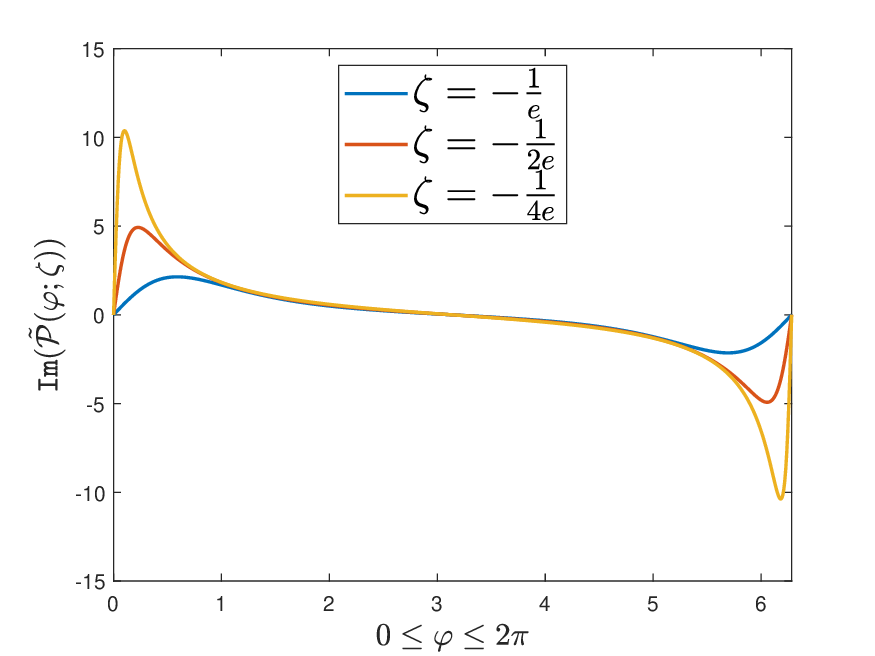}
	 \caption{The graph of \( \mathtt{Re}( \tilde{\mathcal{P}}(\varphi; \zeta) ) \) (Left) and \( \mathtt{Im}( \tilde{\mathcal{P}}(\varphi; \zeta) ) \) (Right)  with \( \varphi \in [0, 2 \pi] \) and \( \zeta = - \tfrac{1}{e}, - \tfrac{1}{2e}, - \tfrac{1}{4e} \)}
	\label{fig:REAL AND IMAGE OF P}
\end{figure} 
From the left subfigure of Figure~\ref{fig:disprove} of the polynomial sequence, it can be observed that the points \(  (0, 2) \) represent a common intersection. Furthermore, the direction of the arrows in the graph indicates the direction of change in the polynomial graph as the value of \( n \) increases. If Buvoli's hypothesis is accurate, then the black marker point \( (-1/e, \tilde{p}_n(-1/e) ) \) in the graph will approach the real axis as \( n \) increases, but will never cross it.

\end{remark}
Rephrasing the above theorem, we obtain the following criterion for determining the conditions that ensure the stability of the scheme based on the given parabolic radius.
\begin{corollary}
Given the parabolic radius \( r_N < r_1 \)  where \( r_1 \) stands for the parabolic radius of the first-order ABTI method and  denoted \( \tilde{\mathcal{P}}(\varphi; \zeta) := \widetilde{ \mathsf{P} }( e^{i \varphi}; \zeta) \) with 
\begin{align*}
\widetilde{ \mathsf{P} }(t; \zeta)  = \frac{ 1 + t }{e^{-\zeta t}-t} - e^{-\zeta t},
\end{align*} 
if for all \( \zeta \in (r_N, 0] \) it holds that
\begin{equation}
 \frac{1}{2 \pi}\int_{0}^{2 \pi}\tilde{\mathcal{P}}(\varphi; \zeta)\cdot e^{-i  N \varphi} \mathrm{d}\varphi > 0,  
 \label{eq:discriminant}
\end{equation}
then ABTI can maintain stability for the approximation accuracy up to \( N \).
\end{corollary}
\begin{corollary}
\label{coro:the smallest modulus zero of the polynomial}
For any given order of approximation \(n\), the parabolic radius \(r_n\) is the smallest modulus zero of the polynomial $\tilde{p}_{n+1}(\zeta; \pi)$. 
The table below lists the parabolic radii for some specific cases. 
\begin{table}[h]
\caption{The parabolic radius for given order of approximation.}
\begin{tabular}{cccccc} 
\hline\noalign{\smallskip}
\( n  \) & \( 4 \) & \( 5 \) & \( 6 \) & \( 7 \) & \( 8 \) \\
\noalign{\smallskip}\hline\noalign{\smallskip}
\( r_{n}  \) & 0.4359 & 0.4110 & 0.4012 & 0.3933 & 0.3882 \\ 
\noalign{\smallskip}\hline
\end{tabular}
\end{table} 
\end{corollary}
%
\subsection{Accuracy loss}
In the previous section, we mentioned that the original ABTI method \cite[section 2.2]{Buvoli2019} fails to achieve the accuracy we envisioned. In this section, we will explore the reasons behind this and provide solutions. First, we examine how ABTI affects the accuracy. If we assume that \( \mathbf{u}^{[n]} \) is computed accurately, it is evident that \( u^{n}\mathbbm{1} \) is also equally accurate through the arithmetic averaging of both sides of the scheme. Thus, \( \mathbf{B}(\alpha)\mathbf{f}^{[n]} \) is the main factor influencing the accuracy of the scheme. The following lemma will characterize the approximation property of the ABTI.
\begin{lemma} For a function \( f(t) \in \mathcal{C} ^{q} (0, T]  \), then exists approximation
\label{lem:approximation}
\begin{equation}
\frac{r}{q}\mathbbm{1}^{\mathrm{T}}\mathbf{B}(\alpha)\mathbf{f}^{[n]} - \int_{t_n}^{t_{n+1}}f(t) \mathrm{d}t = 
\mathcal{O}(\tau^{q}).
\end{equation}
\end{lemma}
\begin{proof}
By invoking the binomial theorem and leveraging the orthogonality of the Fourier basis functions (as given in \cite[Lemma 2.1]{Shen2011}), we obtain the following relation:
\begin{align*}
& \frac{1}{q}\sum_{j=1}^{q}\frac{(\alpha + \omega_j)^q}{q} = \frac{1}{q}\sum_{j=1}^{q}\frac{\sum_{m=0}^{q}\binom{q}{m}\alpha^{q-m}\omega_j^{m}}{q} \\
= &  \frac{1}{q}\sum_{m=0}^{q}\binom{q}{m}\alpha^{q-m}\left(\frac{1}{q}\sum_{j=0}^{q}\omega_j^m\right) =
 \frac{\alpha^{q}}{q} + \frac{1}{q}.  
\end{align*}
We observe that during the initial construction of the scheme, the term \(\frac{1}{q}\) vanishes whenever the number of sampling points \(s\) on the unit disk exceeds the number of Taylor expansion terms \(q\). The presence of this term in the case where \(s = q\) can therefore be identified as the direct cause limiting the ABTI method to sub-optimal accuracy.
Thence, 
\begin{align*}
\frac{r}{q}\mathbbm{1}^{\mathrm{T}}\mathbf{B}(\alpha)\mathbf{f}^{[n]} = & \sum_{\nu=0}^{q-1}\frac{\alpha^{\nu+1}}{\nu+1}  +  \frac{r}{q}\cdot\hat{f}_{q-1} = \int_0^\tau \sum_{\nu=0}^{q-1} \left(\frac{t}{r}\right)^{\nu}\cdot \hat{f}_{\nu} \mathrm{d}t +  \frac{r}{q}\cdot\hat{f}_{q-1} 
\end{align*}
where \( \hat{f}_{\nu} = \left \langle \mathbf{F}\mathbf{f}^{[n]}  \right \rangle _{\nu+1} \).
\par 
By Taylor expansion and its remainder term of an integral, we know that 
\begin{align*}
\hat{f}_{\nu} =  & \frac{1}{q}\sum_{j=1}^{q}\omega_j^{-\nu}f(t_n + r \omega_j) \\
= &  \frac{1}{q}\sum_{j=1}^{q}\omega_j^{-\nu}\left(\sum_{\mu=0}^{q-1}\frac{f^{(\mu)}(t_n)}{\mu!}(r \omega_j)^{\mu} + 	\int_{t_{n}}^{t_{j}^{[n]}}\frac{(t-t_j^{[n]})^{q-1}}{(q-1)!}\cdot f^{(q)}(t)\mathrm{d}t \right) \\
 =  & \sum_{\mu=0}^{q-1}\frac{f^{(\mu)}(t_n)}{\mu!}r^{\mu} \left(\frac{1}{q}\sum_{j=1}^{q}  \omega_j ^{\mu - \nu}\right) + \frac{1}{q}\sum_{j=1}^{q}\omega_j^{-\nu}\left(\int_{t_{n}}^{t_{j}^{[n]}}\frac{(t-t_j^{[n]})^{q-1}}{(q-1)!}\cdot f^{(q)}(t)\mathrm{d}t \right) \\
 =  & \frac{f^{(\nu)}(t_n)}{\nu!}r^{\nu}  + \frac{1}{q}\sum_{j=1}^{q}\omega_j^{-\nu}\int_{t_{n}}^{t_{j}^{[n]}}\frac{(t-t_j^{[n]})^{q-1}}{(q-1)!}\cdot f^{(q)}(t)\mathrm{d}t, \text{ for all \( 0 \le \nu \le q-1 \)}.
\end{align*}
Thence,
\begin{align*}
 &  \frac{r}{q}\mathbbm{1}^{\mathrm{T}}\mathbf{B}(\alpha)\mathbf{f}^{[n]} = \int_{0}^{\tau}\left[\sum_{\nu=1}^{q-1}\frac{f^{(\nu)(t_{n})}}{\nu!}t^{\nu} + \mathfrak{r}_{n,1}(t)\right] \mathrm{d}t \\
 & + \frac{r}{q}\cdot\left(\frac{f^{(q-1)}(t_{n})}{(q-1)!}r^{q-1} + \frac{1}{q}\sum_{j=1}^{q}\omega_{j}\int_{0}^{r}\frac{(r-t)^{q-1}}{(q-1)!}\cdot f^{(q)}(\omega_{j}t+t_{n})\mathrm{d}t\right)
\end{align*}
where
\begin{equation}
\mathfrak{r}_{n,1} := \sum_{\nu=0}^{q-1}\left(\frac{t}{r}\right)^{\nu}\cdot \frac{1}{q}\sum_{j=1}^{q}\omega_{j}^{q-\nu}\int_{0}^{r}\frac{(r-t)^{q-1}}{(q-1)!}\cdot f^{(q)}(\omega_{j}t + t_{n})\mathrm{d}t. \nonumber
\end{equation}
By similar way, 
\begin{align*}
\int_{t_{n}}^{t_{n+1}}f(t)\mathrm{d}t = \int_{0}^{\tau}f(t+t_{n})\mathrm{d}t = \int_{0}^{\tau}\left(
\sum_{\nu=0}^{q-1}\frac{f^{\nu}(t_{n})}{\nu!}t^{\nu} + \mathfrak{r}_{n,2}(t) \right)\mathrm{d}t
\end{align*}
where
\begin{equation}
\mathfrak{r}_{n,2}(t) := \int_{0}^{t}\frac{(t-\sigma)^{q-1}}{(q-1)!}f^{(q)}(\sigma + t_{n})\mathrm{d}\sigma. \nonumber 
\end{equation}
Therefore, we know 
\begin{align*}
& \frac{r}{q}\mathbbm{1}^{\mathrm{T}}\mathbf{B}(\alpha)\mathbf{f}^{[n]} - \int_{t_n}^{t_{n+1}}f(t)\mathrm{d}t \\
= & \int_{0}^{\tau}\mathfrak{r}_{n,1}(t) - \mathfrak{r}_{n,2}(t) \mathrm{d}t + \left(  \frac{f^{(q-1)}(t_{n})}{q!}r^{q} +  \frac{r}{q}\cdot \frac{1}{q}\sum_{j=1}^{q}\omega_{j}\int_{0}^{r}\frac{(r-t)^{q-1}}{(q-1)!}\cdot f^{(q)}(\omega_{j}t+t_{n})\mathrm{d}t\right) \\
= &  \frac{f^{(q-1)}(t_{n})}{\alpha^{q} \cdot q!}\tau^{q} + \mathcal{O}(\tau^{q+1}).
\end{align*}
By this way, we attain a reliable explanation for order-lose phenomenon which completes the proof.
\end{proof}
\begin{remark}
The above lemma tells us that  \( \tfrac{1}{q} \ne 0 \) allows the lower-order terms in the truncation error, specifically \( \frac{f^{(q-1)}(t_n)}{\alpha^q \cdot q!} \tau^q \), to be retained. To achieve the ideal approximation accuracy, we only need to choose the number of sampling points \( s \) to be greater than the number of terms \( q \) in the Taylor expansion. However, increasing the number of sampling points too much will increase the dimensions of the matrices \( \mathbf{A} \) and \( \mathbf{B}(\alpha) \), thus raising the computational cost. Therefore, the optimal accuracy recovery strategy is to set \( s = q + 1 \). Indeed, in this case, the number of unknowns in the solution vector \(\mathbf{u}^{[n]}\) also increases to \(q+1\), which incurs the same computational cost as directly increasing the value of \(q\) in the scenario where \(s = q\). 
\end{remark}
\section{Applications}
\label{sec:Applications}
In this section, we will give applications in  partial differential equations (PDEs).  We consider the linear parabolic equation and its $L^2$-stability fully discrete scheme
\begin{equation}\label{eq:heat_equation}
\begin{cases}
\frac{\mathrm{d}}{\mathrm{d}t} u(\mathbf{x}, t) = \mathcal{L} u(\mathbf{x}, t) + f(t, \mathbf{x}), & \mathbf{x} \in \Omega \times [0,T],  \\
u(\mathbf{x}, 0) = u_0(\mathbf{x}), & \mathbf{x} \in \Omega, \\
u(\mathbf{x}, t) = 0, & \partial \Omega \times [0, T],
\end{cases}
\end{equation}
where \( \mathcal{L} \) is unbounded self-adjoint linear operator in Hilbert space.
\par
It is worth mentioning that when using ABTI for time discretization of the parabolic equation, there exists an issue with the order of space-time discretization. Since the Laplace operator is an unbounded self-adjoint negative semi-definite operator, its spectrum lies on the negative real axis and extends toward negative infinity, we need to first perform spatial discretization of the original parabolic equation, aside from time discretization schemes that are \( A \)-stable or \( A(\theta) \)-stable. Essentially, the discretization of the Laplace operator is equivalent to truncating an unbounded operator into a bounded one. As ABTI is a conditionally stable explicit scheme, spatial discretization must be performed first, and the spatial semi-discretization can be described as a stiff ODE system. The time step constraint, dependent on the CFL condition, is closely related to the spectral radius of the discrete matrix of the Laplace operator. Therefore, the spatial semi-discretized equation can also be generalized as an abstract ODE problem of the form \eqref{eq:ODE}.
\par
The fully discrete scheme of the equation \eqref{eq:heat_equation} gives 
\begin{subequations}\label{eq:fully_discrete}
\begin{align}
\left(\mathbf{E}_{ \tau }\otimes\mathbf{M}\right) \mathbf{u}^{[0]}_{h} =  & \mathbbm{1}\otimes \mathbf{u}_h^{0} + r\left(\mathbf{B}(\alpha)\otimes\mathbf{K}\right) \mathbf{u}^{[0]}_{h} + r \left(\mathbf{B}(\alpha) \otimes \mathbf{E}_{h}\right) \mathbf{f}_{h}^{[0]}, \\ 
\left(\mathbf{E}_{ \tau }\otimes\mathbf{M}\right) \mathbf{u}^{[n+1]}_{h} =  & \left(\mathbf{A}\otimes\mathbf{M}\right)\mathbf{u}_{h}^{[n]} + r \left(\mathbf{B}(\alpha) \otimes \mathbf{K}\right) \mathbf{u}_{h}^{[n]} 
+ r \left(\mathbf{B}(\alpha) \otimes \mathbf{E}_{h}\right) \mathbf{f}_{h}^{[n]},
\end{align}
\end{subequations}
where \( \mathbf{M}\) and \( \mathbf{K} \) are mass matrix and stiffness matrix respectively  associated with finite element methods (FEMs) or another spatial discretization methods.  \( \mathbf{u}_{h}^{[n]}:= [\mathbf{u}_{1,h}^{n}, \mathbf{u}_{2,h}^{n}, \ldots, \mathbf{u}_{q,h}^{n}]^{\mathrm{T}} \) is a \( (q \times \mathrm{N}_h) \)-dimensional column solution vector whose \( N_s \) stand for the spatially discrete degree of freedom. $\mathbf{u}_h^{0} $ is the matrix form of finite element interpolation for the initial value function $u_{0}(\mathbf{x})$, that is the initial finite element solution \( u_h^0 := \sum_{\mu = 1}^{N_h} \langle \mathbf{u}_h^0 \rangle _{\mu} \varphi_{\mu}(\mathbf{x}) \) where \( \langle \mathbf{u}_h^0 \rangle _{\mu}  \) is the \( \mu \)-th component of vector \( \mathbf{u}_{h}^{0} \) and \( \{\varphi_{\mu}(\mathbf{x})\}_{\mu=1}^{N_{h}} \) is the finite element base function. The numerical solution \( u_{h}^{n} \) at the \( n \)-th time level has definitely relation 
\begin{equation}
u_{h}^{n} = \bm{\varphi}^{\mathrm{T}} \mathbf{u}_{h}^{n} = \bm{\varphi}^{\mathrm{T}} \left(\tfrac{1}{q}\mathbbm{1}_{\tau}^{\mathrm{T}}\otimes \mathbf{E}_{h}\right)\mathbf{u}_{h}^{[n]} \nonumber
\end{equation}
where \( \bm{\varphi} =[\ldots, \varphi_j, \ldots]^{\mathrm{T}} \) is   the FEM base function vector, \( \mathbbm{1}_{\tau} \) is the all-one vector associated with temporal discrete scale \( q \) and \( \mathbf{E}_{h} \) is the unite matrix associated with spatial discrete scale or so-called DoF( Degree of Freedom) of FEM \( N_{h} \).  
\subsection{\( L^2 \)-stability}
\begin{theorem}[ \( L^{2} \)-stability]
\label{thm:l2_stability}
For arbitrary \( q \ge 1 \), under the time scale restriction with respect to the parabolic CFL condition \( \rho(\mathbf{G}) < 1 \) where \( \mathbf{G}:= \mathbf{A}\otimes\mathbf{E}_h  + \tau (\mathbf{B}\otimes\mathbf{K})(\mathbf{E}_\tau\otimes\mathbf{M})^{-1} \) with a brief note \( \mathbf{B}:=\tfrac{1}{\alpha}\mathbf{B}(\alpha) \), then the fully discrete scheme \eqref{eq:fully_discrete} is \( L^2 \) stable in the sense that   
\begin{equation}
\left\| u_h^{n+1} \right\| \le \left\| u_h^0 \right\| + \tau \sum_{\nu=0}^{n}\left\| f_{h}^{\nu} \right\|. 
\end{equation}

\end{theorem}
\begin{proof}
We review the spatial semi-discretization
\begin{align*}
\frac{\mathrm{d}}{\mathrm{d}t} \mathbf{M}\mathbf{u}_h(t) = \mathbf{K}\mathbf{u}_h(t) + \mathbf{f}_{h}(t),
\end{align*}
where \( \mathbf{M} := \int_{\Omega}\bm{\varphi}\cdot \bm{\varphi}^{\mathrm{T}}\mathrm{d}x,  \mathbf{K} := \int_{\Omega} \bm{\varphi}\cdot \Delta\bm{\varphi}^{\mathrm{T}}\mathrm{d}x = - \int_{\Omega} \nabla\bm{\varphi}\cdot (\nabla\bm{\varphi})^{\mathrm{T}}\mathrm{d}x, \mathbf{f}_h(t) := \int_{\Omega}\bm{\varphi}\cdot f \mathrm{d}x \) are mass matrix and  stiffness matrix and load vector respectively. 
\par 
The fully discretization, denote \( \mathbf{v}_h^{[n]} := (\mathbf{E}_\tau\otimes\mathbf{M})\mathbf{u}_h^{[n]} \) and \( \mathbf{G}:= \mathbf{A}\otimes\mathbf{E}_h  + r (\mathbf{B}\otimes\mathbf{K})(\mathbf{E}_\tau\otimes\mathbf{M})^{-1} \) whose the size of identity matrix \( \mathbf{E}_{ \tau } \) and \(  \mathbf{E}_{ h } \) keep up with time and space discrete matrix respectively, 
\begin{align*}
\mathbf{v}_h^{[n+1]} = \mathbf{G}\mathbf{v}_h^{[n]} + \tau \tilde{\mathbf{f}}_h^{[n]}, \text{ where } 
\tilde{\mathbf{f}}_h^{[n]} := \left(\mathbf{B} \otimes \mathbf{E}_h\right) \mathbf{f}_{h}^{[n]}. 
\end{align*}
By the recursion, the preceding relation subjects to 
\begin{equation}
\mathbf{v}_h^{[n+1]} = \mathbf{G}^{n+1}\mathbf{v}_h^{0} + \tau \sum_{\nu=0}^{n}\mathbf{G}^{n-\nu}\tilde{\mathbf{f}}_h^{\nu}. \nonumber
\end{equation}
Notice that
\begin{equation}
\frac{1}{q}(\mathbf{u}_h^{[m]})^{\mathrm{T}}\mathbf{v}_h^{[n]} = \frac{1}{q}(\mathbf{u}_h^{[m]})^{\mathrm{T}}(\mathbf{E}_\tau\otimes\mathbf{M})\mathbf{u}_h^{[n]} = \frac{1}{q}\sum_{j=1}^{q}\left \langle \mathbf{u}_{h,j}^{[m]}, \mathbf{u}_{h,j}^{[n]} \right \rangle = \left \langle u_h^{m}, u_h^{n} \right \rangle, \nonumber
\end{equation}
 we take the product \( \frac{1}{q}(\mathbf{u}_h^{[n+1]})^{\mathrm{T}} \) on the both side that 
\begin{align*}
\lVert u_h^{n+1} \rVert ^2 \le  &  \left \langle u_h^{n+1,(n+1)}, u_h^0 \right \rangle + \tau \sum_{\nu=0}^{n}\left \langle u_h^{n+1, (n-\nu)}, \tilde{f}_h^{\nu}  \right \rangle  \\
\le  &  \lVert u_h^{n+1, (n+1)} \rVert \cdot \lVert u_h^0 \rVert + \tau \sum_{\nu=0}^{n}\lVert u_h^{n+1, (n-\nu)}\rVert \cdot \lVert \tilde{f}_h^{\nu} \rVert,
\end{align*}
where the \( u_h^{n+1, (n-\nu)}  \) has similar definition to numerical solution \( u_{h}^{n} \) associated with vector \(  (\mathbf{G}^{n-\nu})^{\mathrm{T}}\mathbf{u}_{h}^{[n+1]}  \)  that
\begin{equation}
u_h^{n+1, (n-\nu)} := \bm{\varphi}^{\mathrm{T}} \left(\frac{1}{q}\mathbbm{1}_{\tau}^{\mathrm{T}}\otimes \mathbf{E}_{h}\right) \left((\mathbf{G}^{n-\nu})^{\mathrm{T}}\mathbf{u}_{h}^{[n+1]}\right), ~ \text{ for all \( -1 \le \nu \le n \)}. \nonumber
\end{equation}
The CFL condition associated with the spectral radius \( \rho( \mathbf{G} ) \le 1 \) yields 
\begin{align*}
\lVert u_{h}^{n+1,(n-\nu)} \rVert \le \lVert u_h^{n+1} \rVert \text{ for all }  -1 \le \nu \le n. 
\end{align*}
Notice that \( \lVert \tilde{f}_h^{\nu} \rVert \sim \lVert f_h^{\nu} \rVert \), we therefore obtain the upper bound estimate 
\begin{align*}
\lVert u_h^{n+1} \rVert \le  \lVert u_h^0 \rVert + \tau \sum_{\nu=0}^{n}\lVert \tilde{f}_h^\nu \rVert \le  \lVert u_h^0 \rVert + \tau \sum_{\nu=0}^{n}\lVert f_h^\nu \rVert.
\end{align*}
The proof of the theorem concludes at this point.
\end{proof}
The above \( L^2 \)-stability analysis framework can expand to block Adams-Bashforth (BAB), block Adams-Moulton (BAM) or block backward differentiation formula (BBDF) which refers to \cite{Buvoli2018} for temporal discretizations within general form as 
\begin{equation}
\mathbf{u}^{[n+1]} = \mathbf{Au}^{[n]} + \tau \mathbf{B}\left(\mathcal{L}_h\mathbf{u}^{[n]} + \mathbf{f}^{[n]} \right), \nonumber
\end{equation}
where \( \mathbf{u}^{[n]}, \mathbf{A} \) and \( \mathbf{B} \) deduced by specific strategy and \( \mathcal{L}_h \) is the discretization of \( \mathcal{L} \).  Particularly, our analysis results have two following reduced version. 
\begin{itemize}
\item 
When \( r \to 0 \) or \( \alpha \to \infty \), the discrete matrices \( \mathbf{A}\to 1, \mathbf{B} = \frac{1}{\alpha}\mathbf{B}(\alpha) \to 1 \). Meanwhile the above method method reduce to the classical AB1 method. 
\item
When \( \mathbf{M}\to 1, \mathbf{K} \to \lambda \), the above method means that \( \Delta u \to \lambda u \) that is the heat equation reduces to the Dahlquist test problem.  
At this point, the CFL condition for \( L^2 \)-stability is consistent with the stability condition derived from the stability region of the ODE problem.
\end{itemize}
\subsection{\( L^{2} \)-error}
The error bound of fully discrete scheme, combined with the spatial finite element discretization, essentially only requires the analysis of the standard spatial semi-discretization scheme, the \( L^2 \)-stability analysis results Theorem \ref{thm:l2_stability}, and the approximation of the ABTI method  Lemma \ref{lem:approximation}.
\begin{theorem}[\( L^2 \)-error]
\label{thm:p1_time_error}
Denote the error \( e^{n} := \lVert u(t_n) - u^n \rVert \) associated with varying discrete time nodal, then fully discrete scheme \eqref{eq:fully_discrete} has the error bound 
\begin{equation}
e^{N} \le C(\tau^{q-1} + h^{k}), \nonumber
\end{equation}
where \( q \) is the discretization scale used by the ABTI method, and \( k \) is the accuracy of spatial finite element discretization.
\end{theorem}
\begin{proof}
Denote 
\begin{align*}
\text{ \( L^2 \) projection } P_h & : V \to V_h \text{ such that } \left(P_h u - u, \chi \right) = 0, \\
\text{ Ritz projection } R_h & : H_0^1 \to V_h \text{ such that } A\left(R_h u - u, \chi \right) = 0,
\end{align*}
where  \( \chi \in \mathcal{V}_h \subset \mathcal{V} \).
The discrete Laplacian operator defined by \( \left(\Delta_h u_h , \chi \right) = A\left(u_h, \chi\right) \) for all \( \chi \in \mathcal{V}_h \). 
Noticing the \( L^2 \) projection acting on the original equation, we can compare with
\begin{align*}
\frac{\mathrm{d}}{\mathrm{d}t} P_h u - \Delta_h R_h u = P_h f, \text{ since \(  P_h \Delta = \Delta_h R_h \)} 
\end{align*}
with the spatial semi-discretization
\begin{equation}\label{eq:spatial semi-discrete}
\frac{\mathrm{d}}{\mathrm{d}t} u_h(t) -\Delta_h u_h(t) = P_h f.
\end{equation}
and obtain  the error equation of spatial semi-discretization that 
\begin{equation}\label{eq:spatial semi-discrete}
\frac{\mathrm{d}}{\mathrm{d}t} e_h(t) - \Delta_h e_h(t) = - \frac{\mathrm{d}}{\mathrm{d}t} (P_h u - u)  +  \Delta_h (R_h u - u) , \text{ with  initial value } e_h(0) = 0,
\end{equation} 
where \( e_h(t) := u(t) - u_h(t) \). 
\par 
By standard analysis techniques \cite[Theorem 1.2]{Thomee2006}, the error bound of the spatial semi-discretization \eqref{eq:spatial semi-discrete} yields 
\begin{equation}
\lVert e_h(t) \rVert = \lVert u_h(t_n) - u(t_n) \rVert  \le C _{h} h^{k} . \nonumber
\end{equation}
Next, we consider the temporal discretization  about \eqref{eq:spatial semi-discrete} as
\begin{align*}
\mathbf{u}_h^{[n+1]} = \mathbf{A}\mathbf{u}_h^{[n]} + r \mathbf{B}(\alpha)\left(\Delta_h \mathbf{u}_h^{[n]} + \mathbf{f}_h^{[n]} \right),
\end{align*}
where \( \mathbf{f}_h^{[n]} := P_h \mathbf{f}^{[n]}  \). Take \( \frac{1}{q}\mathbbm{1}^{\mathrm{T}} \) on the both side,
we know
\begin{align*}
u_h^{n+1} = u_h^{n} + \frac{r}{q}\mathbbm{1}^{\mathrm{T}} \mathbf{B}(\alpha)\left(\Delta_h \mathbf{u}_h^{[n]} + \mathbf{f}_h^{[n]} \right).
\end{align*}
Take \( \int_{t_n}^{t_{n+1}}\cdot ~ \mathrm{d}t  \) for \( \frac{\mathrm{d}}{\mathrm{d}t} u_h(t) = \Delta_h u_h(t) + f_h \) with notation \(
f_h := P_h f \), we notice
\begin{align*}
u_h(t_{n+1}) = u_h(t_{n}) + \int_{t_n}^{t_{n+1}}  \left(\Delta_h u_h(t) + f_h \right) \mathrm{d}s.
\end{align*}
Thence,  we know
\begin{align*}
e_{h, \tau}^{n+1} = e_{h, \tau}^{n} +  \frac{r}{q}\mathbbm{1}^{\mathrm{T}} \mathbf{B}(\alpha)\left(\Delta_h \mathbf{u}_h^{[n]} + \mathbf{f}_h^{[n]} \right) - \int_{t_n}^{t_{n+1}}  \left(\Delta_h u_h(t) + f_h \right) \mathrm{d}s =: e_{h, \tau}^{n} + T_n,
\end{align*}
where	\( e_{h, \tau}^{n}:= u_h^n - u_h(t_n) \).  The approximation Lemma \ref{lem:approximation} tell us that \( e_{h, \tau}^{n+1} \le e_{h, \tau}^{n} + \mathcal{O}(\tau^{q-1}) \)
then \( e_{h, \tau}^{n+1} \le e_{h, \tau}^{0} + \sum_{j=0}^{n} T_n \).
\par 
Therefore, we know
\begin{align*}
\lVert u_h^n - u(t_n) \rVert \le  & \lVert u_h^n - u_h(t_n) \rVert + \lVert u_h(t_n) - u(t_n) \rVert  \\
\le  & \left( \lVert e_{h, \tau}^{0}  \rVert  + \sum_{j=0}^{n} \lVert T_n  \rVert \right) +  \lVert u_h(t_n) - u(t_n) \rVert  \\
\le  & C(\tau^{q-1} + h^{k}),
\end{align*}
where \( C = \max \left\{C _{ \tau }, C _{ h } \right\} \).
\end{proof}

\section{Numerical Verifications}
\label{sec:Numerical Verifications}
This section will present several typical numerical examples to validate the analysis results mentioned earlier.
\subsection{ODE}
To avoid the influence of spatial errors, we first consider a stiff nonlinear ODE problem to testing the convergence of ABTI. In general, such problems do not have an analytical solution. However, the Allen-Cahn ODE problem falls into this category and has an exact solution \cite{Stuart1998}
\begin{equation}
u(t) = \frac{u^0}{ \sqrt{e^{- \frac{2}{ \varepsilon^2}t} + u_0^2(1-e^{- \frac{2}{ \varepsilon^2 }t})}}. \nonumber
\end{equation}
The equation is expressed as follows
\begin{equation}
\frac{\mathrm{d}}{\mathrm{d}t} u + \frac{1}{ \varepsilon^2 }f(u) = 0,   t \in [0,T] \text{ with initial data } u(0) = u^0, \nonumber
\end{equation}
where \( f(u) = u^3 - u \) equals to the first derivative of double-well potential function \( F(u) = \tfrac{1}{4}(u^2- 1)^2 \) associated with \( u \). 
\par
Theo. c.o. is the abbreviation for theoretical convergence order. \( \tau \) represents the length of discrete time step, and \( \mathcal{E}(\tau) \) denotes error at the terminal time $T = 1$ under this time discretization. In the entire comparative experiment, we fix the small parameter in the equation as \( \varepsilon = 0.5 \), which reflects the width of the interface and the initial data $u^{0} = 0.01$. 
For the parameter \( \alpha \) in the scheme, since different values of \( \alpha \) do not affect the step size constraint, we without loss of generality choose \( \alpha = 1 \). Since we are dealing with a nonlinear problem, the initial solution vector requires solving the following nonlinear algebraic equation
\begin{equation}
\mathbf{u}^{[0]} = u^{0}\mathbbm{1} - \frac{r}{\epsilon^{2}} \mathbf{B}(\alpha)f(\mathbf{u}^{[0]}). \nonumber
\end{equation}
In our experiment, we use the Newton's iteration for numerically solving the above equation, with the initial guess for the iteration set as \( u^{0}\mathbbm{1} \).
\par 
The  Table \ref{tab:AC_ODE_qs} and \ref{tab:AC_ODE_qs1} compare the errors and convergence for the cases \( s = q \) and \( s = q+1 \). From the experimental results, we achieved the optimal convergence of the ABTI method with lower computational cost. Due to the limitations of computer machine precision, it is recommended to use the high-precision toolbox\footnote{\url{https://www.advanpix.com/}} in MATLAB for schemes of order four and above.
\begin{table}[h]
\caption{The error and convergence of the original ABTI $s=q$ in solving the Allen-Cahn ODE.}
\label{tab:AC_ODE_qs}
\begin{tabular}{ccccccc}
\toprule
\multirow{2}{*}{$1/\tau$} &
\multicolumn{2}{c}{$q=1$} &
\multicolumn{2}{c}{$q=2$} &
\multicolumn{2}{c}{$q=3$} \\
\cline{2-7} 
& {$\mathcal{E}(\tau)$} & {c.o.} & {$\mathcal{E}(\tau)$} & {c.o.} & {$\mathcal{E}(\tau)$} & {c.o.} \\
\midrule
  128  &  2.299e-02 & - &  2.011e-02 & - &  1.548e-04 & - \\ 
  256  &  2.669e-02 & -0.215 &  1.024e-02 & 0.974 &  3.941e-05 & 1.974 \\ 
  512  &  2.861e-02 & -0.100 &  5.162e-03 & 0.988 &  9.939e-06 & 1.987 \\ 
 1024  &  2.958e-02 & -0.048 &  2.592e-03 & 0.994 &  2.496e-06 & 1.994 \\ 
Theo. c.o. & - &  0.000 & - & 1.000  &  - & 2.000 \\
\bottomrule
\end{tabular}
\end{table}
\begin{table}[h]
\caption{The error and convergence of the modified ABTI $s=q+1$ in solving the Allen-Cahn ODE.}
\label{tab:AC_ODE_qs1}
\begin{tabular}{ccccccc}
\toprule
\multirow{2}{*}{$N _{\tau}$} &
\multicolumn{2}{c}{$q=1$} &
\multicolumn{2}{c}{$q=2$} &
\multicolumn{2}{c}{$q=3$} \\
\cline{2-7} 
& {$\mathcal{E}(N _{\tau})$} & {c.o.} & {$\mathcal{E}(N _{\tau})$} & {c.o.} & {$\mathcal{E}(N _{\tau})$} & {c.o.} \\
\midrule
  128  &  2.054e-02 & - &  1.687e-04 & - &  4.885e-07 & - \\ 
  256  &  1.034e-02 & 0.990 &  4.115e-05 & 2.035 &  5.287e-08 & 3.208 \\ 
  512  &  5.188e-03 & 0.995 &  1.016e-05 & 2.018 &  6.113e-09 & 3.112 \\ 
 1024  &  2.598e-03 & 0.998 &  2.523e-06 & 2.009 &  7.337e-10 & 3.059 \\ 
Theo. c.o. & - &  1.00 & - & 2.00  &  - & 3.00 \\
\bottomrule
\end{tabular}
\end{table}
\subsection{PDE}
Consider classical heat equation 
\begin{equation}
\begin{cases}
\frac{\mathrm{d}}{\mathrm{d}t} u = \Delta u, & u \in [0, T] \times [-\pi, \pi], \\
u(0, x) = \cos(x), & x \in [ - \pi, \pi] \text{ and periodical boundary condition}.
\end{cases}
\end{equation}
 As we well known that its exact solution \( u = e^{-t}\cos(x) \) and  obviously \( u \in H_0^{1}([-\pi, \pi]) \).
\par 
In high-dimensional complex regions, the mass and stiffness matrices obtained from high-order finite element methods are relatively complicated. Therefore, to better understand the optimality of the stability conditions derived from the amplification matrix, we first examine the content in the case of a one-dimensional homogeneous boundary condition. Considering the relation between FEM with the linear elements approximation on \( 1 \)-dimensional case and  the center difference methods, the stiffness matrix  for the discretization of Laplacian operator can represented by 
\begin{align*}
\mathbf{K} = 
\frac{1}{h^2}
\begin{bmatrix}
-2 & 1 & & \\
1 & \ddots & \ddots & \\
 & \ddots & \ddots & 1 \\
& & 1 & -2
\end{bmatrix}
\end{align*}
and the mass matrix $\mathbf{M} = \mathbf{E}_h $ is a identity matrix. 
By simple evaluation, we claim \( \mathbf{K} \) is a negative defined matrix and the spectral radius  \( \rho(\mathbf{K}) \le \tfrac{4}{h^2} \). Meanwhile, a key tensor identity will be used, namely  \( \rho(\mathbf{G}) =  \rho(\mathbf{A}\otimes \mathbf{E}_{h} + \tau \mathbf{B}\otimes \mathbf{K}) = \rho(\mathbf{A} - \tau \rho(\mathbf{K}) \mathbf{B}) \).
It is easy to check that \( \rho(\mathbf{G}) < 1 \) whenever \( \tfrac{\tau}{h^2} \le \tfrac{r_n}{4} \). 
Thus, for a given target time approximation accuracy, the determination of the parabolic CFL condition only requires the discriminant \eqref{eq:discriminant}  to determine the parabolic radius \( r_n \).
\par
For example, to illustrate with a fixed first-order time approximation accuracy, we have \( r_1 = 0.6 \). At this point, we fix the spatial step size \( h = \tfrac{\pi}{32} \) and choose \( \tau = \frac{r_1 \times h^2}{4} \) and \( \tau = \frac{r_1 \times h^2}{4} (1\pm 0.1) \), and compare their numerical solutions to test the stability of the scheme. Set the parameter of ABTI \( \alpha = 1 \). 
The table below shows the spectral radius of the amplification matrix under different conditions
and Figure \ref{fig:blowup} shows the corresponding numerical solution. It is easy to see that when the time step size is outside the range constrained by the CFL condition, the numerical solution will exhibit a blow-up phenomenon.
\begin{table}
\caption{The spectral radius of the amplification matrix with different $\tau$}
\begin{tabular}{cccc} 
\hline\noalign{\smallskip}
\( \tau  \) & \( 0.9 \times \tfrac{r_2 h^2}{4}  \) & \( \times \tfrac{r_2 h^2}{4}  \) & \( 1.1 \times \tfrac{r_2 h^2}{4}  \)  \\
\noalign{\smallskip}\hline\noalign{\smallskip}
\( \rho(\mathbf{G})  \) & 0.9996 & 0.9995 &1.1161  \\ 
\noalign{\smallskip}\hline
\end{tabular}
\end{table}
\begin{figure}[h]
\centering
\includegraphics[scale=.2]{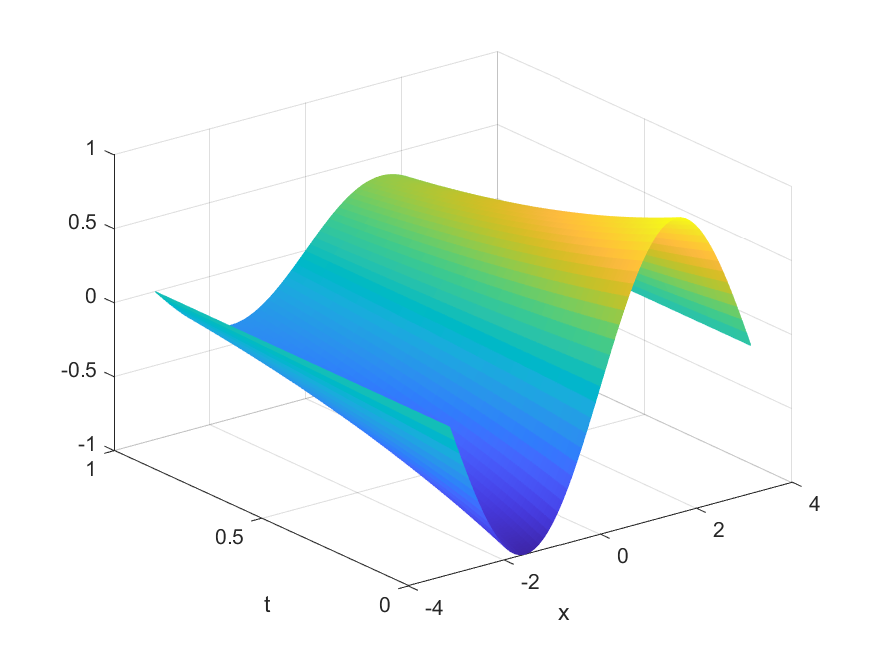}
\includegraphics[scale=.2]{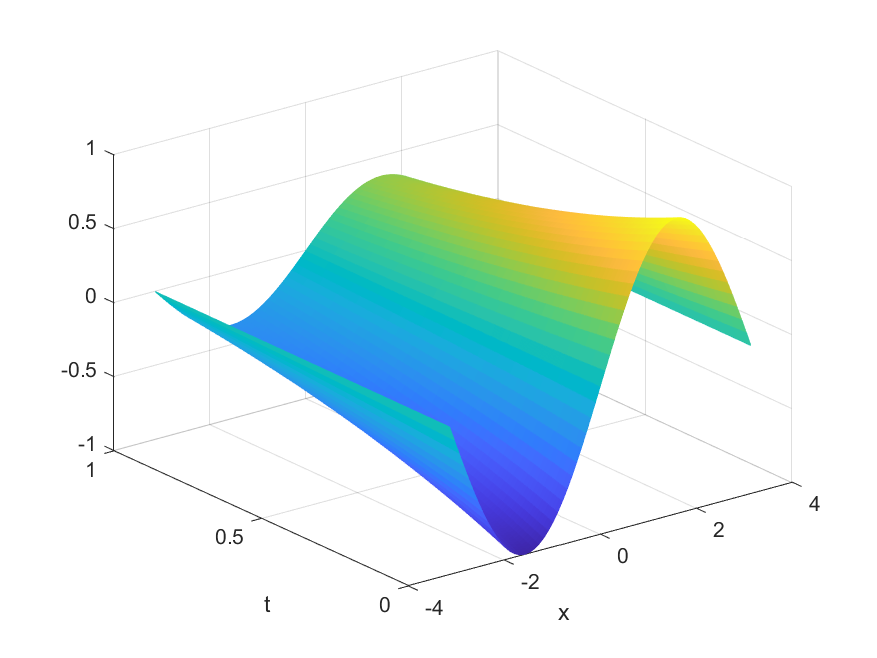}
\includegraphics[scale=.2]{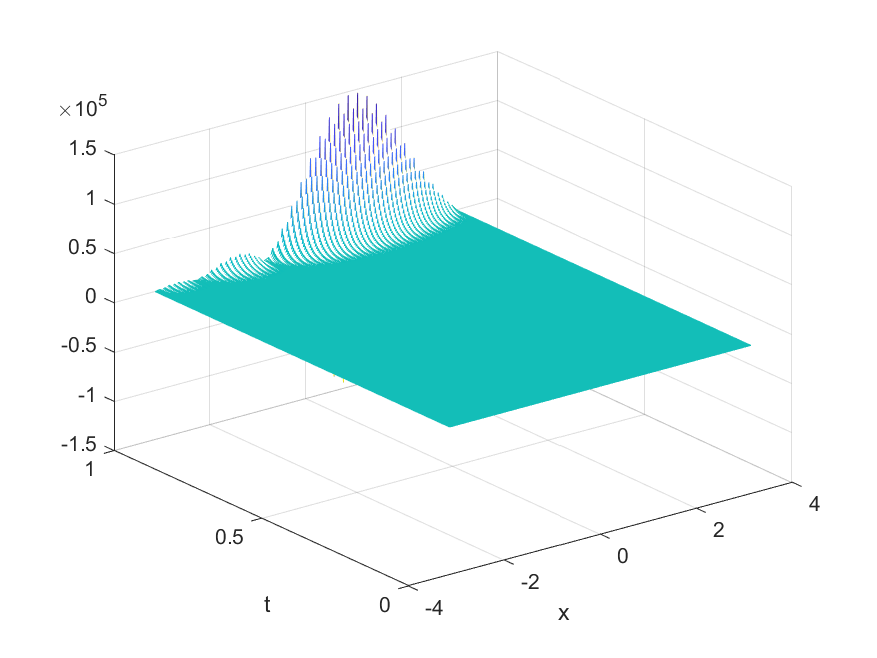}
\caption{\( \tau= 0.9 \times \tfrac{r_1 h^2}{4} \) (Left), \( \tau= \tfrac{r_1 h^2}{4} \) (Middle) and \( \tau =  1.1 \times \tfrac{r_1 h^2}{4} \) (Right). }
\label{fig:blowup}
\end{figure}

\section{Conclusion}
In this paper, we address a conjecture regarding the uniform stability of Adams-Bashforth-type integrator with arbitrary-order accuracy. We provide the discriminant conditions for maintaining scheme stability at a given parabolic radius and the allowable accuracy for this stability. We also identify the reasons why the original scheme fails to achieve the ideal accuracy and propose solutions. As an application of the algorithm to parabolic equations, we derive the CFL condition compatible with ODE problems, along with its \( L^2 \)-stability and error bounds. Finally, some numerical examples are presented at the end of the paper to validate our analytical results.
\appendix
\section{Derivation of Characteristic Polynomial}
\label{sec:Derivation of Characteristic Polynomial}
Derivation of characteristic polynomial of matrix \( \mathbf{A}+ z \mathbf{B} \) includes two steps. First, a trick of the determinant of the block matrix deduces the equivalent matrix in the sense of characteristic polynomial. Second, the virtue of a constructed matrix helps us to obtain a difference equation about the determinant, and its written matrix form can carry out whose specific expression. In addition, the notation \( z \) used in the lemmas and proofs in this section simply represents a complex variable, which is distinct from \( z = \lambda \tau \in \mathbb{C} \).
\begin{lemma}\label{lem:same_characteristic_polynomial}
For all \( \alpha \in \mathbbm{R}^{+} \) and \( z \in \mathbb{C} \), matrices $\mathbf{A}+z\mathbf{S}(\alpha)\mathbf{F}$ and $\mathbf{e}_{1}\mathbf{e}_{1}^{\mathrm{T}} + z \mathbf{F}\mathbf{S}(\alpha)$ have the same characteristic polynomial. Here, \( \mathbf{e}_1 \) stand for a vector whose only the first element equals one and vanish else, and length keep up with the square matrix \( \mathbf{FS}(\alpha) \).
\end{lemma}
\begin{proof}
By the standard operation, we consider the characteristic polynomial in start
\begin{equation}
p_q(\lambda;\mathbf{A}+z\mathbf{B}(\alpha)) := 
\begin{vmatrix}
\mathbf{A} + z\mathbf{S}(\alpha)\mathbf{F} -\lambda \mathbf{E}
\end{vmatrix}\nonumber
\end{equation}
and then effort to obtain its expression. According to the unitary property of Fourier matrix $\sqrt{q}\mathbf{F}$, that is $\mathbf{F}\mathbf{F}^{*}=\frac{1}{q} \mathbf{E}=\mathbf{F}^*\mathbf{F}$, we obtain 
 \begin{align*}
\begin{vmatrix}
\mathbf{A} + z\mathbf{S}(\alpha)\mathbf{F} -\lambda \mathbf{E}
\end{vmatrix}
 = & \begin{vmatrix}
q \mathbf{A}\mathbf{F}^{*}  + z\mathbf{S}(\alpha)  - q \lambda \mathbf{F}^{*}
\end{vmatrix}
\cdot
\begin{vmatrix}
\mathbf{F} 
\end{vmatrix} \\
 = & \begin{vmatrix}
q\mathbf{F}\mathbf{A}\mathbf{F}^{*}  + z \mathbf{F}\mathbf{S}(\alpha)  - q\lambda \mathbf{F} \mathbf{F}^{*} 
\end{vmatrix} \\
= & \begin{vmatrix}
\mathbf{e}_{1}\mathbf{e}_{1}^{\mathrm{T}} + z \mathbf{F}\mathbf{S}(\alpha)  -\lambda \mathbf{E} 
\end{vmatrix}
\end{align*}
where \( \mathbf{e}_{1} := [1,0, \ldots, 0]^{\mathrm{T}} \) and \( \mathbf{F}^{*} \) is the conjugate transpose of \( \mathbf{F} \). 
Therefore,
\begin{equation}
p_q(\lambda;\mathbf{A}+z\mathbf{S}(\alpha)\mathbf{F}) = p_q(\lambda;\mathbf{e}_{1}\mathbf{e}_{1}^{\mathrm{T}} + z \mathbf{F}\mathbf{S}(\alpha)) \nonumber
\end{equation}
finished this proof.
\end{proof}
 By this operation, we can get two benefits that $\mathbf{e}_{1}\mathbf{e}_{1}^{\mathrm{T}}$ is extremely sparse matrix with only one non-trivial element at first-column first-row and $\mathbf{F}\mathbf{S}(\alpha)$ is a real quasi-upper triangle matrix with non-zero lower sub-diagonal elements which will be investigated below. 
\begin{lemma}\label{lem:two_inverse_matrix}
For all $0\le j \le q$,
\begin{enumerate}[1)]
\item 
\begin{equation}
\begin{bmatrix}
\beta_{0} & 0 & \ldots & 0 \\ 
\beta_{1} & \beta_{0} & \ldots & 0 \\ 
\vdots & \ddots & \ddots & \vdots \\
\beta_{q} & \ldots & \beta_{1} & \beta_{0}
\end{bmatrix}^{-1}
=
\begin{bmatrix}
\beta_{0}^{(-1)} & 0 & \ldots & 0 \\ 
\beta_{1}^{(-1)} & \beta_{0}^{(-1)} & \ldots & 0 \\ 
\vdots & \ddots & \ddots & \vdots \\
\beta_{q}^{(-1)} & \ldots & \beta_{1}^{(-1)} & \beta_{0}^{(-1)}
\end{bmatrix}
\label{eq:equlity_one}
\end{equation}
where $ \beta_{j} = \gamma_{j}(-\alpha z) $ and $ \beta_{j}^{(-1)} = \gamma_{j}(\alpha z) $.
\item 
\begin{equation}
\begin{bmatrix}
1 && &&\\ 
\lambda\beta_{0}^{(-1)} & 1 &&&\\ 
\vdots & \ddots & \ddots &  &  \\
\lambda\beta_{q-1}^{(-1)} & \ldots & \lambda\beta_{0}^{(-1)} & 1
\end{bmatrix}^{-1}
= 
\begin{bmatrix}
\eta_{0} && &&\\ 
\eta_{1} & \eta_{0} & &&\\ 
\vdots & \ddots & \ddots &  &  \\
\eta_{q} & \ldots & \eta_{1} & \eta_{0}
\end{bmatrix}
\end{equation}
where $ \eta_{j} := \Sigma_{k=0}^{j-1}\gamma_{k}(j-k)(-\lambda)^{j-k}(\alpha z)^{k} $ for all $0 \le j \le q$ with setting \( \eta_{0} :=  \Sigma_{k=0}^{-1} =  1\). In fact, the value $\eta_{0}=1$ can also be derived from the property that the inverse of a unit lower triangular matrix remains a unit lower triangular matrix. 
\end{enumerate}
\end{lemma}
\begin{proof}
The inverse matrix derived from the identity of a lower triangular matrix is always invertible and remains a lower triangular matrix with an identity. The correctness of the lemma can be verified by confirming that the product of a matrix and its inverse yields the identity matrix. In fact, the proof of both relations ultimately comes down to the binomial theorem.
\begin{enumerate}[1)]
\item 
It is easy to check for the equality \eqref{eq:equlity_one} that
\begin{align*}
\sum_{\nu=0}^{i-j}\beta_{\nu}\beta_{i-j-\nu}^{(-1)} = & \sum_{\nu=0}^{i-j}\gamma_{\nu}(-\alpha z)\gamma_{i-j-\nu}(\alpha z) = \gamma_{i-j}(\alpha)\sum_{\nu=0}^{i-j}\binom{i-j}{\nu}(-1)^{\nu}(1)^{i-j-\nu} \\
= & \gamma_{i-j}(\alpha \cdot 0) =
\begin{cases}
1, \text{ when \( j = i \),}\\
0, \text{ otherwise. }
\end{cases}
\end{align*}
\item 
Let $\beta_{-1}^{(-1)}:=1/\lambda$ for the compatible which need to be considered. 
While $i=j$, $\lambda \beta_{-1}^{(-1)}\eta_{0} = 1$ obviously. The same result can be achieved by stating that the inverse of the identity lower triangular matrix must be the identity lower triangular matrix. Next, we consider $i-j\ge 1$,
\begin{align*}
& \sum_{\nu=0}^{i-j}\lambda\beta_{\nu-1}^{(-1)}\eta_{i-j-\nu} = \eta_{i-j} + \sum_{\nu=1}^{i-j-1}\lambda\beta_{\nu-1}^{(-1)}\eta_{i-j-\nu} + \lambda \beta_{i-j-1}^{(-1)}\eta_{0}   \\
 = & 
\sum_{k=0}^{i-j-1}\gamma_{k}(i-j-k)(-\lambda)^{i-j-k}(\alpha z)^{k} 
+ \\
& \lambda\sum_{\nu=1}^{i-j-1}\gamma_{\nu-1}(\alpha z) \sum_{k=0}^{i-j-\nu-1}\gamma_{k}(i-j-\nu-k)(-\lambda)^{i-j-\nu-k}(\alpha z)^{k} \\
& + \lambda \gamma_{i-j-1}(\alpha z) \\
 = & (-\lambda)^{i-j}\sum_{k=0}^{i-j-2}\gamma_{k}\left((i-j-k) \frac{-\alpha z}{\lambda} \right)  \\
& - (-\lambda)^{i-j}\sum_{\nu=1}^{i-j-1}\gamma_{\nu-1}\left(\frac{-\alpha z}{\lambda}\right) \sum_{k=0}^{i-j-\nu-1}\gamma_{k}\left((i-j-\nu-k)\frac{-\alpha z}{\lambda} \right). 
\end{align*}
While $i=j+1$, the above equality equals zeros for which the summation will be empty.
\par While \( i \ge j +2 \), let \( H:=i-j-2 \ge 0 \) and \( x:= - \tfrac{z}{\lambda} \) as brief notations,
\begin{align*}
 & \sum_{k=0}^H \gamma_{k}\left((H-k+2)x \right)  -  \sum_{\nu=0}^H \gamma_{\nu}(x) \sum_{k=0}^{H-\nu}\gamma_{k}\left((H-\nu-k+1)x \right) \\
    =  & \sum_{k=0}^H \gamma_{k}(H-k+2)x^{k} -  \sum_{\nu=0}^H \sum_{k=0}^{H-\nu}\frac{(H-\nu-k+1)^k}{k!\cdot \nu!} x^{\nu+k} \\
    =  & \sum_{\mu=0}^H \gamma_{\mu}(H-\mu+2)x^{\mu} -  \sum_{\mu=0}^H \left(\sum_{k=0}^{\mu} \frac{(H-\mu+1)^k}{k!\cdot (\mu-k)!}\right) x^{\mu}, \text{ let $\mu = \nu+k$} \\
    = & \sum_{\mu=0}^{H}\left(\gamma_{\mu}(H-\mu+2)  - \sum_{k=0}^{\mu} \frac{(H-\mu+1)^k}{k!\cdot (\mu-k)!}\right)x^{\mu} \\
= & \sum_{\mu=0}^{H}\left(\gamma_{\mu}(H-\mu+2)  - \frac{1}{\mu!}\sum_{k=0}^{\mu}\binom{\mu}{k}(H-\mu+1)^{k}1^{\mu-k}\right)x^{\mu} \\
= & {\sum_{\mu=0}^{H}\left(\gamma_{\mu}(H-\mu+2)  - \frac{(H-\mu+2)^{\mu}}{\mu!}\right)x^{\mu} = \sum_{\mu=0}^{H}0\cdot x^{\mu} = 0}. 
\end{align*}
\end{enumerate}
At this point, the proof of this lemma is complete.
\end{proof}
\begin{lemma}\label{lem:simplify_eig_polynomial}
For all \( \alpha \in \mathbb{R} \) and \( z, \lambda \in \mathbb{C} \), then 
\begin{equation}
\sum_{j=0}^{q}\gamma_{q-j}(\alpha z)\sum_{k=0}^{j-1}\gamma_{k}((j-k)\alpha z)(-\lambda)^{j-k} = \sum_{j=0}^{q}\gamma_{q-j}((j+1)\alpha z)(-\lambda)^{j}. 
\end{equation}
If $q=0$, then $\Sigma_{k=0}^{-1} = 1$ on the left-hand side of the above equality for setting.
\end{lemma}
\begin{proof}
To improve readability, we use the matrix language to represent two different summation methods and their relationship visually. Denote \(A_j := \gamma_{q-j}(\alpha z), C_{j-k}^{(j)} := \gamma_{k}((j-k)\alpha z) \) 
\begin{equation}\label{eq:matrix_to_sum}
\mathbbm{1}^{\mathrm{T}}
\begin{bmatrix}
A_1 &&& \\
& A_2 && \\
&& \ddots & \\
&&& A_q
\end{bmatrix}
\cdot
\begin{bmatrix}
C_{1}^{(1)} & &&& \\
C_{1}^{(2)} & C_{2}^{(2)} & &&\\
\vdots & \vdots & \ddots \\
C_{1}^{(q)} & C_{2}^{(q)} & \ldots & C_{q}^{(q)} 
\end{bmatrix}
\begin{bmatrix}
(-\lambda)^{1} \\
(-\lambda)^{2} \\
\vdots \\
(-\lambda)^{q} \\
\end{bmatrix}.
\end{equation}
In the following, $\mathbf{LHS}$ and $\mathbf{RHS}$ abbreviate the left-hand side and right-hand side respectively. 
\begin{align*}
\mathbf{LHS} = & \sum_{j=1}^{q}A_j \sum_{k=0}^{j-1}C_{j-k}^{(j)}(- \lambda)^{j-k} + \gamma_{q}(\alpha z),  \text{ evaluate from right to left in eq. \eqref{eq:matrix_to_sum}}\\ 
= & \sum_{j=1}^{q}\left(\sum_{k=j}^{q}A_k C_{j}^{(k)}\right)(-\lambda)^{j}  + \gamma_{q}(\alpha z), \text{ evaluate from left to right in eq. \eqref{eq:matrix_to_sum}} \\
= & \sum_{j=1}^{q}\left(\sum_{k=j}^{q}\gamma_{q-k}(\alpha z)\cdot \gamma_{k-j}(j \alpha z)\right)(-\lambda)^{j}  + \gamma_{q}(\alpha z) \\
= & \sum_{j=1}^{q}\left(\sum_{k=j}^{q} j^{k-j}\binom{q-j}{q-k}\right)\cdot \gamma_{q-j}(\alpha z)(-\lambda)^{j}  + \gamma_{q}(\alpha z)  \\
= & \sum_{j=0}^{q}\gamma_{q-j}((j+1)\alpha z)(-\lambda)^{j} = \mathbf{RHS},
\end{align*}
where the last line uses a fact that \( \Sigma_{k=j}^{q}j^{k-j}\tbinom{q-j}{q-k} = (j+1)^{q-j} \text{ for all \( 1 \le j \le q \)}  \) which is easy to check by binomial theorem and \( j=0 \) can include the term \( \gamma_{q}(\alpha z) \).
\end{proof}
\begin{lemma}\label{lem:characteristic_polynomial_evaluation}
For all \( z \in \mathbb{C} \) and \( q \ge 1 \), 
\begin{equation}
 p_q(\lambda; \mathbf{e}_1 \mathbf{e}_1 ^{\mathrm{T}} + z \mathbf{FS}(\alpha)) = f_q( \lambda; \alpha z) + f _{q-1}( \lambda; \alpha z) - \gamma_{q}(-z), \nonumber
\end{equation} 
where \( f_q( \lambda; \alpha z) :=  \sum_{j=0}^{q} \gamma_{q-j}((j+1)\alpha z)(-\lambda)^{j} \).
\end{lemma}
\begin{proof}
 By simple evaluation, we can know that  
\begin{align*}
\langle \mathbf{FS}(\alpha)\rangle_{j,k} = & \sum_{n=1}^{q}\mathbf{F}_{j,n}\mathbf{S}_{n,k}(\alpha) =  \frac{1}{q}\sum_{n=1}^{q}\omega_n^{1-j}\cdot\frac{(\alpha+\omega_n)^k}{k}
= \frac{1}{q}\sum_{n=1}^{q}\omega_n^{1-j}\cdot\frac{1}{k}\sum_{m=0}^{k}\binom{k}{m}\alpha^{k-m}\omega_{n}^{m} \\
= & \sum_{m=0}^{k}\frac{1}{k}\binom{k}{m}\alpha^{k-m}\cdot\frac{1}{q} \sum_{n=1}^{q} \omega_n^{1-j+m} \\
= & 
\begin{cases}
\frac{1}{k}\binom{k}{j-1}\alpha^{k-j+1} + \frac{1}{q}\delta_{\{(j,k),(1,q)\}} &\text{ if } j \le k+1,\\
0 , &\text{ if } j > k+1,
\end{cases}
\end{align*}
where should be used a fact that $\tfrac{1}{q}\Sigma_{n=1}^{q} \omega_n^{1-j+m} = 1 $ whenever $1-j+m = 0 ~ (\mod q)$ and varnishing otherwise.
\begin{align*}
\begin{cases}
0 \le m = j -1 \le k \le q & \Rightarrow 1 \le j \le q, k \ge j -1. \\
0 \le m = j -1 + q \le k \le q & \Rightarrow \added{ j = 1}, k = q.
\end{cases}
\end{align*}
 Then,  $\mathbf{FS}(\alpha)$ is a real quasi-upper triangle matrix with a non-zero lower sub-diagonal elements.
Denote 
\begin{equation}
m_{j,k} = 
\begin{cases}
z\mu_{j,k} & \text{ when } 1 \le j < k \le q, \\
z\mu_{j,k} - \lambda & \text{ when } 1 \le j = k \le q. \\
\end{cases}
\text{ with } \mu_{j,k} := \frac{1}{k}\binom{k}{j-1}\alpha^{k-j+1}, \nonumber 
\end{equation}
then \( p_q(\lambda; \mathbf{e}_1 \mathbf{e}_1 ^{\mathrm{T}} + z \mathbf{FS}(\alpha)) \) equals to
\begin{align*}
 & 
\begin{vmatrix}
m_{1,1} + 1 & m_{1,2} & \ldots & m_{1,q-1} & m_{1,q} + \frac{z}{q} \\
z & m_{2,2}  & \ldots & m_{2,q-1} & m_{2,q} \\
  & z/2 & \ddots & m_{3,q-1} & m_{3,q} \\
  &   & \ddots & \vdots & \vdots \\
&  &  & z/(q-1) & m_{q,q}
\end{vmatrix} \\
= & 
\begin{vmatrix}
m_{1,1} & m_{1,2} & m_{1,3} & \ldots & m_{1,q} \\
z & m_{2,2} & m_{2,3} & \ldots & m_{2,q} \\
 & z/2 & m_{3,3} & \ldots & m_{3,q} \\
 &  & \ddots & \ddots & \vdots \\
 &  &  & & m_{q,q}
\end{vmatrix}
+ 
\begin{vmatrix}
m_{2,2} & m_{2,3} & m_{2,4} & \ldots & m_{2,q} \\
z/2 & m_{3,3} & m_{3,4} & \ldots & m_{3,q} \\
 & z/3 & m_{4,4} & \ldots & m_{4,q} \\
 & & \ddots & \ddots & \vdots \\
 & &  &  & m_{q,q}
\end{vmatrix}
- \gamma_{q}(-z) \\
=: & D_{q} + \tilde{D}_{q} - \gamma_{q}(-z).
\end{align*}
The blank part of the matrix is filled with zero elements. Here, we just give the operation of determinant $D_q$ while $\tilde{D}_q$ can be calculated by similar way. 
Expand $D_q$ along with the last line and then get a recursion relation
\begin{align*}
D_q = & m_{q,q}D_{q-1} - \frac{z}{q-1}\left(m_{q-1,q}D_{q-2} - \frac{z}{q-2}\left(\ldots\right)\right) \\
= & m_{q,q}D_{q-1} + \frac{-zm_{q-1,q}}{q-1}D_{q-2} + \frac{z^2m_{q-2,q}}{(q-2)_{2}}D_{q-3} + \ldots  + \frac{(-z)^{q-2} m_{2,q}}{(2)_{q-2}}D_1 + \frac{(-z)^{q-1} m_{1,q}}{(1)_{q-1}} \\
= & \sum_{j=0}^{q-1}\frac{(-z)^{j}}{(q-j)_{j}}m_{q-j,q}D_{q-j-1} = - \sum_{j=0}^{q-1}\frac{(-z)^{j+1}}{(q-j)_{j}} \frac{1}{q}\binom{q}{q-j-1}\alpha^{j+1} D_{q-j-1}  - \lambda D_{q-1} \\
 = & - \sum_{j=0}^{q-1}\gamma_{j+1}(-\alpha z)D_{q-j-1} - \lambda D_{q-1}, 
\end{align*}
where $(a)_n := a(a+1)\ldots(a+n-1), (a)_0 := 1 $ are the rising factorial (or called Pochhammer symbol) and  \( D_{0} = 1 \) for setting. 
Alternatively, we can also obtain the same result directly using the related formula about Hessenberg matrix.
\par 
Thus, we can obtain a difference equation associated with the determinant \( D_{j} \), 
\begin{equation}
\lambda D_{q-1}  + \sum_{j=0}^{q}\gamma_{j}(-\alpha z)D_{q-j} = 0. \nonumber
\end{equation}
Denoting \( \beta_{j}:= \gamma_{j}(-\alpha z) \), the previous difference equations formulate a matrix equation about unknown vector \( \mathbf{d} \) 
\begin{equation}
\bm{\Lambda}\mathbf{d} + \mathbf{B}\mathbf{d}=\mathbf{e}_{1}, \nonumber
\end{equation}
where 
\begin{equation}
\bm{\Lambda} := 
\begin{bmatrix}
0 &  & & \\
\lambda & 0  & & \\
 & \ddots  & \ddots &  \\
& & \lambda & 0
\end{bmatrix}, 
\mathbf{B} := 
\begin{bmatrix}
\beta_{0} & & & & \\ 
\beta_{1} & \beta_{0} & & & \\ 
\vdots & \ddots & \ddots & & \\
\beta_{q} & \ldots & \beta_{1} & \beta_{0}
\end{bmatrix}
\text{ and }
\mathbf{d}:=
\begin{bmatrix}
D_0 \\
D_1 \\
\vdots \\
D_q
\end{bmatrix}. \nonumber
\end{equation}
Thence, 
\begin{equation}
\bm{\Lambda}\mathbf{d} + \mathbf{B}\mathbf{d}=\mathbf{e}_{1}
\Rightarrow 
(\mathbf{B}^{-1}\bm{\Lambda} + \mathbf{E})\mathbf{d} = \mathbf{B}^{-1} \mathbf{e}_{1} 
\Rightarrow 
\mathbf{d} = (\mathbf{B}^{-1}\bm{\Lambda} + \mathbf{E})^{-1} \mathbf{B}^{-1} \mathbf{e}_{1}.
 \nonumber
\end{equation}
Firstly, thank to the property of Toeplitz matrix with  \( \beta_{j} = \gamma_{j}(-\alpha z) \), \( \mathbf{B}^{-1} \) also possess Toeplitz structure and each entries have expression 
\begin{equation}
\mathbf{B}^{-1} = 
\begin{bmatrix}
\beta_{0}^{(-1)} & &&\\ 
\beta_{1}^{(-1)} & \beta_{0}^{(-1)} &&\\ 
\vdots & \ddots & \ddots & & \\
\beta_{q}^{(-1)} & \ldots & \beta_{1}^{(-1)} & \beta_{0}^{(-1)}
\end{bmatrix}
\text{ where \( \beta_{j}^{(-1)} = \gamma_{j}(\alpha z) \) } \text{( by Lemma \ref{lem:two_inverse_matrix})}.  \nonumber
\end{equation}
%
Secondly, \( \mathbf{H}=(\mathbf{B}^{-1}\bm{\Lambda} + \mathbf{E})^{-1} \) is a Teoplitz matrix similarly, 
\begin{equation}
\mathbf{H} := 
\begin{bmatrix}
\eta_{0} && &&\\ 
\eta_{1} & \eta_{0} & &&\\ 
\vdots & \ddots & \ddots &  &  \\
\eta_{q} & \ldots & \eta_{1} & \eta_{0}
\end{bmatrix}
\text{ where } \eta_{n} = \sum_{j=0}^{n-1}\gamma_{j}(n-j)(-\lambda)^{n-j}(\alpha z)^{j} \text{ ( by Lemma \ref{lem:two_inverse_matrix})}\nonumber
\end{equation}
where \( 0 \le n \le q \) and let \( \Sigma_{j=0}^{-1} := 1 \).
Since  \( \beta_{j}^{(-1)} \) and \( \eta_{j} \) for \( 0 \le j \le q \) are specific, immediately we can obtain the explicit expression of  \( D_{q} = \Sigma_{j=0}^{q}\beta_{q-j}^{(-1)}\cdot\eta_{j} \), that is, 
\begin{align}
D_q = \sum_{j=0}^{q}\gamma_{q-j}(\alpha z) \sum_{k=0}^{j-1}\gamma_{k}(j-k)(-\lambda)^{j-k}(\alpha z)^{k} = \sum_{j=0}^{q}
\gamma_{q-j}((j+1)\alpha z)(-\lambda)^{j}, \nonumber
\end{align}
where the Lemma \ref{lem:simplify_eig_polynomial} should be used. By similar operation, 
\begin{equation}
\tilde{D}_{q}= \sum_{j=0}^{q-1}\gamma_{q-1-j}((j+1)\alpha z)(-\lambda)^{j}. \nonumber
\end{equation}
Therefore, we obtain the explicit expression of the \(  p_q(\lambda; \mathbf{e}_1 \mathbf{e}_1 ^{\mathrm{T}} + z \mathbf{FS}(\alpha)) \).
\end{proof}
Based on Lemmas \ref{lem:same_characteristic_polynomial} and \ref{lem:characteristic_polynomial_evaluation}, it is not difficult to conclude that
\begin{equation}
p_q(\lambda;\mathbf{A}+z\mathbf{B}) = p_q(\lambda;\mathbf{A}+z\mathbf{B}(\alpha)/\alpha) = f_q( \lambda; z) + f _{q-1}( \lambda; z) - \gamma_{q}(-z/\alpha), \nonumber
\end{equation}
where \( f_q( \lambda; z) :=  \sum_{j=0}^{q} \gamma_{q-j}((j+1)z)(-\lambda)^{j} \).
\par 
To intuitively demonstrate the accuracy of the above derivation, we can plot the stability region using the original matrix \( \mathbf{A} + z \mathbf{B}(\alpha)/\alpha \), the equivalent matrix \( \mathbf{e}_1 \mathbf{e}_1 ^{\mathrm{T}} + z \mathbf{FS}(\alpha)/\alpha \), and the characteristic polynomial in Figure \ref{fig:uniform_stability_area}.
\begin{figure}[t]
\centering
\includegraphics[scale=.4]{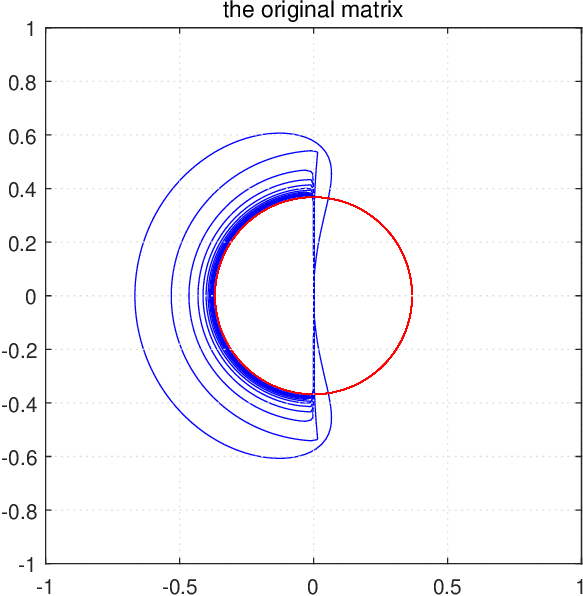}
\includegraphics[scale=.4]{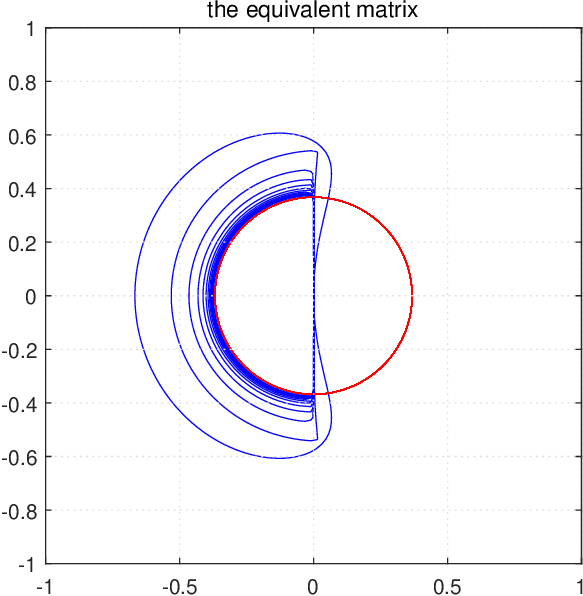}
\includegraphics[scale=.4]{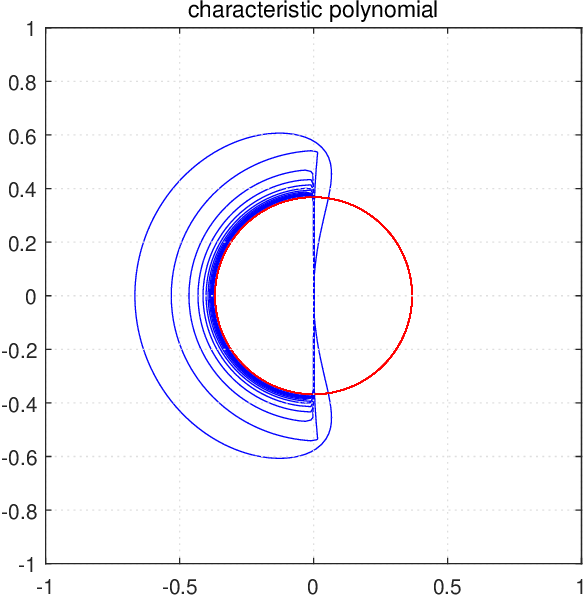}
\caption{The stability domain described by the spectral radius of the original matrix $\mathbf{A} + z \mathbf{B}(\alpha)/\alpha = \mathbf{A} + z \mathbf{S}(\alpha)\mathbf{F}/\alpha$ (Left), the spectral radius of the equivalent matrix $\mathbf{e}_{1}\mathbf{e}_{1}^{\mathrm{T}} + z \mathbf{F}\mathbf{S}(\alpha)/\alpha$ (Middle) and the maximum module of the zeros of the derived characteristic polynomial (Right). The circle centered at the origin with a radius of \( 1/e \) marked with red line and the contour of stability region for $q = 2,\ldots, 16$ with blue one. }
\label{fig:uniform_stability_area}
\end{figure}


%


\begin{acknowledgements}
This work is partially supported by the National Natural Science Foundation of China No.12171385. 
\end{acknowledgements}
\section*{Conflict of interest}
The authors declared that they have no conflicts of interest to this work.

\bibliographystyle{spmpsci}      
\bibliography{ref_AB}   


\end{document}